\documentclass{amsart}

\usepackage{mathrsfs}
\usepackage{color}
\usepackage{hyperref}
\usepackage[latin1]{inputenc}

\DeclareMathOperator{\res}{res}

\newtheorem{theorem}{Theorem}[section]
\newtheorem{lemma}[theorem]{Lemma}

\newtheorem{proposition}[theorem]{Proposition}
\newtheorem{remark}[theorem]{Remark}
\newtheorem{definition}[theorem]{Definition}
\newtheorem{example}[theorem]{Example}

\numberwithin{equation}{section}

\newcommand{\bz}{{\mathbb B}}
\newcommand{\cz}{{\mathbb C}}

\newcommand{\nz}{{\mathbb N}}
\newcommand{\rz}{{\mathbb R}}

\newcommand{\calC}{\mathcal{C}}
\newcommand{\calH}{\mathcal{H}}
\newcommand{\calK}{\mathcal{K}}

\newcommand{\scrC}{\mathscr{C}}
\newcommand{\scrD}{\mathscr{D}}
\newcommand{\scrE}{\mathscr{E}}
\newcommand{\scrL}{\mathscr{L}}
\newcommand{\scrS}{\mathscr{S}}

\newcommand{\ulA}{\underline{A}}
\newcommand{\ulwhA}{\underline{\widehat{A}}}
\newcommand{\ulwhB}{\underline{\widehat{B}}}
\newcommand{\ulE}{\underline{\mathscr{E}}}
\newcommand{\ulwhE}{\underline{\widehat{\mathscr{E}}}}

\newcommand{\comp}{\mathrm{comp}}
\newcommand{\dbar}{d\hspace*{-0.08em}\bar{}\hspace*{0.1em}}
\newcommand{\eps}{\varepsilon}
\newcommand{\forget}[1]{}

\newcommand{\lra}{\longrightarrow}

\newcommand{\re}{\mathrm{Re}\,}

\newcommand{\spk}[1]{\langle#1\rangle}
\newcommand{\wh}{\widehat}
\newcommand{\wt}{\widetilde}

\begin{document}
\title[Bounded $H_\infty$-calculus for cone differential operators]%
{Bounded $H_\infty$-calculus for\\ cone differential operators}

\author{E. Schrohe}
\address{Leibniz Universit\"at Hannover, Institut f\"ur Analysis, Hannover (Germany)}
\email{schrohe@math.uni-hannover.de}

\author{J.\ Seiler}
\address{Universit\`a di Torino, Dipartimento di Matematica, Torino (Italy)}
\email{joerg.seiler@unito.it}

\begin{abstract}
We prove that parameter-elliptic extensions of cone differential operators have a bounded $H_\infty$-calculus. 
Applications concern the  Laplacian and the porous medium equation on manifolds with warped conical singularities. 
\end{abstract}

\maketitle
\section{Introduction}\label{sec:intro}

We show that closed extensions of differential operators on manifolds with conical 
singularities, which are parameter-elliptic with respect to a sector 
 $$\Lambda=\Lambda(\theta)=\left\{re^{i\varphi}:
     r\ge 0,\;\theta\le\varphi\le 2\pi-\theta\right\},\qquad 0<\theta<\pi,$$
admit a bounded $H_\infty$-calculus in the natural weighted 
$L_p$-Sobolev spaces, $1<p<+\infty$; 
see Theorem \ref{thm:hinfty} for the precise formulation. 
To this end we combine our  investigations on this subject in \cite{CSS1}, \cite{CSS2}, \cite{RoidosSchr}, \cite{SchrSe}  
with the results of Gil, Krainer and Mendoza \cite{GiMe}, \cite{GKM2}, \cite{GKM}. 
Our main analytic tool is the calculus of parameter-dependent cone pseudodifferential operators, 
cf. Schulze \cite{Schu1}, \cite{Schu2}; for a concise summary we refer the reader to the appendix of \cite{SchrSe}.
The ellipticity conditions require the invertibility of both the interior symbol and the conormal symbol as well as 
resolvent estimates for the model cone operator, see conditions (E1), (E2), (E3) in Section \ref{sec:03}. 

Let $\bz$ be a compact smooth manifold with boundary $X=\partial\bz$; the dimension of $X$ is denoted by $n$. 
We shall identify a collar neighborhood of the boundary with $[0,1)\times X$ 
and denote by $t$ the variable of $[0,1)$. 
A \textit{cone differential operator} $A$ of order $\mu\in\nz$ is a $\mu$-th order differential operator with smooth 
coefficients on the interior of $\bz$ and a specific structure in the collar neighborhood, namely, 
\begin{equation}\label{eq:intro01}
 A=t^{-\mu}\sum_{k=0}^\mu a_k(t)(-t\partial_t)^k\quad \text{with }a_k\in\scrC^\infty([0,1)\mathrm{Diff}^{\mu-k}(X)),
\end{equation}
where $\mathrm{Diff}^{j}(X)$ denotes the Fr\' echet space of differential operators of order at most $j$ on $X$. 
In general, we will assume all operators to act on sections 
of vector bundles over $\bz$, but for simplicity we do not indicate the vector bundles in the notation.  

We will consider a closed extension $\ulA$ of $A$, considered as an unbounded operator 
\begin{equation}\label{eq:intro02}
 A:\scrC^{\infty}_{\comp}(\mathrm{int}\,\bz)\subset\calH^{s,\gamma}_p(\bz)\longrightarrow 
 \calH^{s,\gamma}_p(\bz)
\end{equation}
in a weighted $L_p$-Sobolev space  $\calH^{s,\gamma}_p(\bz)$ of functions on $\bz$ of smoothness $s$ and 
weight $\gamma\in \rz$, defined in Section \ref{sec:2.1.2}.

Suppose that $\Lambda$ is a sector of minimal growth, i.e., for $R$ sufficiently large,
$$\sup_{\lambda\in\Lambda,\,|\lambda|\ge R}
\|\lambda (\lambda-\ulA)^{-1}\|_{\scrL(\calH^{s,\gamma}_p(\bz))}<+\infty.$$ 
Writing $\ulA_c:=\ulA+cI$ with $c\ge R$  one can then define 
the operator $f(\ulA_c)$  by 
$$f(\ulA_c) = \frac 1{2\pi i} \int_{\partial \Lambda} f(\lambda) (\lambda -\ulA_c )^{-1}\, d\lambda$$ 
for every bounded holomorphic function $f:\cz\setminus\Lambda\to\cz$ $($actually, the above Dunford integral 
makes sense only for functions $f$ decaying with some positive power rate at infinity; the case of general $f$ 
involves an approximation argument$)$. One says that $\ulA_c$ admits a bounded $H_\infty$-calculus, 
if $ \|f(\ulA_c)\|_{\scrL(\calH^{s,\gamma}_p(\bz))}\le C\|f\|_{\infty}$ for a constant $C$ independent of $f$;  
see \cite{DHP} or  \cite{KW} for more details.
The $H_\infty$-calculus and the related notion of maximal regularity play an 
essential role in the analysis of non-linear parabolic evolution equations in the functional-analytic approach based 
on semi-group theory.  

A prototype of a cone differential operator, of importance in many applications, is the Laplacian with respect 
to a conically degenerate  metric, i.e. a Riemannian metric in the interior of $\bz$ which in the collar neighborhood 
is of the form 
$$ g= dt^2 + t^2 h(t)$$
with a smooth family $h(t)$, $0\le t\le 1$, of Riemannian metrics on $X$. 
One speaks of a straight conical metric, if $h$ is independent of $t$, otherwise of a warped conical metric. 
We find sufficient conditions for an extension $\underline\Delta$ of the (warped) Laplacian to satisfy the above 
assumptions (E1), (E2) and (E3), see Theorem \ref{thm:5.3}. 
Finally, we outline how these results can be used to show the existence of a short time solution to the porous 
medium equation on manifolds with warped conical singularities, thus improving on earlier work in \cite{RoidosSchrPME} 
for straight  cones. 

The paper is structured as follows. 
After recalling some notation in Section \ref{sec:notation} we describe, in Section \ref{sec:02}, the closed extensions 
of an elliptic cone differential operator $A$ as in \eqref{eq:intro01} and a  relation between the closed extensions of 
$A$ and those of the so-called \textit{model cone operator} $\wh{A}$ associated with $A$, 
\begin{equation}\label{eq:intro03}
 \wh{A}=t^{-\mu}\sum_{k=0}^\mu a_k(0)(-t\partial_t)^k,
\end{equation}
which is a differential operator on $X^\wedge:=(0,+\infty)\times X$. 
This relation was first introduced in \cite{GKM}; we provide here an alternative, equivalent description. 
In Section \ref{sec:03} we explain the concept 
of parameter-ellipticity and show that the resolvent is an element of the cone calculus with parameters, 
see Theorem \ref{thm:resolvent}. 
This improves the results in \cite{SchrSe}, where the coefficients $a_k$ of $A$ were required to be independent 
of $t$ for small $t$ and the domain was assumed to be dilation invariant.  
In Section \ref{sec:04} we prove resolvent estimates in all Sobolev spaces 
$\calH_p^{s,\gamma}(\bz)$ of order $s\ge0$ and the existence of the bounded $H_\infty$-calculus. 
We rely on techniques developed in \cite{RoidosSchr};
note, however, that in \cite[Section 3]{RoidosSchr} detailed knowledge about the structure of the resolvent in 
the spirit of Theorem \ref{thm:resolvent} was assumed while we find here simple conditions which guarantee 
precisely this. 
In Section \ref{sec:applications} we discuss the closed extensions of the warped Laplacian and  the porous medium 
equation, establishing in Theorem \ref{thm:solution} the existence of a unique short-time solution. 

\section{Some basic notation}\label{sec:notation}

Throughout the paper we write $\spk{y}=(1+|y|^2)^{1/2}$ for $y\in \rz^n$. 
\subsection{Function spaces}

Let $\bz$ as described in the introduction, $\gamma,\rho\in\rz$ and $p\in(1,+\infty)$. 
Write $X^\wedge=(0,+\infty)\times X$ with $X=\partial\bz$ and local coordinates $(t,x)$.

\subsubsection{Function spaces on the infinite cone}

For $s\in\nz_0$ let $\calH^{s,\gamma}_p(X^\wedge)$ denote the space of all $u\in H^s_{p,\mathrm{loc}}(X^\wedge)$ 
such that 
 $$t^{\frac{n+1}{2}-\gamma}(t\partial_t)^k \partial^\alpha_x u(t,x)\in L^p\Big(X^\wedge,\frac{dt}{t}dx\Big),
     \qquad k+|\alpha|\le s,$$
where $n$ is the dimension of $X$. 
These are Banach spaces in a natural way. By interpolation and duality one can also extend the definition to $s\in\rz$.

Let $U_1,\ldots,U_N$ be a covering of $X$ together with coordinate maps $\kappa_j:U_j\to V_j\subset\rz^n$. 
On $\rz\times X$  consider the coordinate maps $\wh\kappa_j:\wh{U}_j:=\rz\times U_j\to \rz^{n+1}$ given by 
$\wh\kappa_j(t,x)=\big(t,\spk{t}\kappa_j(x)\big)$.
We can define Sobolev spaces $H^{s,\rho}_p(\rz\times X)_{cone}$ 
on $\rz\times X$ in the standard way: Taking a partition of unity $\phi_1,\ldots,\phi_N$ on $X$ subordinate to the 
covering $U_1,\ldots,U_N$ we ask that $(\phi_ju)\circ\wh{\kappa}_j^{-1}$ belongs to 
$H^{s,\rho}_p(\rz^{n+1})=\spk{\cdot}^{-\rho}H^{s}_p(\rz^{n+1})$.  

The Banach space $\calK^{s,\gamma}_p(X^\wedge)^\rho$ consists of all $u$ such that, with an arbitrary cut-off function 
$\omega(t)\in\scrC^\infty_{\mathrm{comp}}([0,+\infty))$\footnote{i.e., $\omega\equiv1$ near $t=0$ and $\omega$ 
has compact support}, 
 $$\omega u\in \calH^{s,\gamma}_p(X^\wedge),\qquad (1-\omega)u\in H^{s,\rho}_p(\rz\times X)_{cone}.$$
For convenience we write 
$\calK^{s,\gamma}_p(X^\wedge):=\calK^{s,\gamma}_p(X^\wedge)^0$. 
Also we introduce 
\begin{align*}
 \scrS^{\gamma}_0(X^\wedge)&=\mathop{\mbox{\Large$\cap$}}_{\substack{s,\rho\ge0\\ k\in\nz_0}}
     (\log t)^{-k}\calK^{s,\gamma}_p(X^\wedge)^\rho,\\
 \scrS^{\gamma}_\eps(X^\wedge)&=\mathop{\mbox{\Large$\cap$}}_{\substack{s,\rho\ge0\\ 0\le\delta<\eps}}
     \calK^{s,\gamma+\delta}_p(X^\wedge)^\rho\qquad(\eps>0).
\end{align*}
These are Fr\' echet spaces and the definition is independent of the choice of $p$.
Note that $(1-\omega) u\in\scrS(\rz,\scrC^\infty(X))$ is rapidly decreasing for $t\to+\infty$ whenever 
$u\in\scrS^{\gamma}_\eps(X^\wedge)$ with $\eps\ge0$.  

\subsubsection{Function spaces on $\bz$}\label{sec:2.1.2}
Using the above spaces on $X^\wedge$ and a cut-off function $\omega\in\scrC^\infty_{\mathrm{comp}}([0,1))$,  
we introduce the Banach spaces $\calH^{s,\gamma}_p(\bz)$ by requiring that  
 $$\omega u\in{\calK^{s,\gamma}_p(X^\wedge)},\qquad (1-\omega)u\in{H^s_{p}(\bz)},$$
while $\scrC^{\infty,\gamma}_\eps(\bz)$, $\eps\ge0$, is defined by the requirement 
 $$\omega u\in{\scrS^{\gamma}_\eps(X^\wedge)},\qquad (1-\omega)u\in{\scrC^\infty(\bz)};$$
here, $1-\omega$ is considered as a smooth function on $\bz$ supported away from the boundary.

\subsection{Fr\' echet space valued symbols}

Let $E$ be a Fr\' echet space and $\Sigma$  a sector in $\rz^2$. We write $S^\mu(\Sigma,E)$ 
for the space of all smooth functions $a:\mathrm{int}\,\Sigma\to E$ such that 
for every multi-index $\alpha\in\nz_0^2$ and every continuous semi-norm $p$ on $E$
 $$p\big(D^\alpha_\eta a(\eta)\big)\le C_{\alpha,p}\spk{\eta}^{\mu-|\alpha|}, \quad \eta\in \Sigma.$$

\subsection{Twisted operator-valued symbols}\label{sec:twisted}

We define a group $\kappa_s$, $s>0$, of operators on $\scrC^\infty_{\mathrm{comp}}(X^\wedge)$ by 
 $$(\kappa_su)(t,x)=s^{-(n+1)/2}u(st,x).$$
They extend to unitary operators on $L^2(X^\wedge,t^ndtdx)$ and to continuous operators on $\scrD^\prime(X^\wedge)$.

Given two Banach spaces  $E,F\subset\scrD^\prime(X^\wedge)$ 
which are invariant under every $\kappa_s$, $s>0$, we denote by $S^\mu(\Sigma;E,F)$ the space of all 
smooth functions $a:\mathrm{int}\,\Sigma\to\scrL(E,F)$ such that 
 for every multi-index $\alpha\in\nz_0^2$
 $$\|\kappa_{\spk{\eta}}^{-1}\big(D^\alpha_\eta a(\eta)\big)\kappa_{\spk{\eta}}\|_{\scrL(E,F)}\le     
C_\alpha\spk{\eta}^{\mu-|\alpha|}, \quad \eta\in \Sigma.$$

If $F=\mathop{\mbox{\Large$\cap$}}_{j\in\nz}F_j$ is a projective limit of Banach spaces $F_j\subset\scrD^\prime(X^\wedge)$ which are all 
invariant under $\kappa_s$, $s>0$, we set 
$S^\mu(\Sigma;E,F)=\mathop{\mbox{\Large$\cap$}}_{j\in\nz}S^\mu(\Sigma;E,F_j)$. 
It carries a natural Fr\' echet topology. 

\subsection{Parameter-dependent cone pseudodifferential operators}

Throughout the paper we shall make use of Schulze's calculi for $($parameter-dependent$)$ pseudodifferential operators 
on $\bz$. We will use various subclasses of these calculi like $C^\mu(\bz;(\gamma,\gamma-\mu,(-N,0]))$ and 
$C^\mu_{O}(\Sigma)$. For a concise presentation of the parameter-dependent classes we refer the reader to the 
appendix of \cite{SchrSe}.

\section{Closed extensions of cone differential operators}\label{sec:02}

Being a differential operator on the interior of $\bz$, we can associate with $A$ its 
\textit{homogeneous principal symbol} 
$\sigma_\psi^\mu(A)\in\scrC^\infty(T^*\mathrm{int}\,\bz\setminus0)$. 
The limit 
  $$\wt{\sigma}_\psi^\mu(A)(x,\xi,\tau):=\lim_{t\to 0}t^\mu\sigma_\psi^\mu(A)(t,x,t^{-1}\tau,\xi)$$
defines the so-called \textit{rescaled principal symbol} 
$\wt{\sigma}_\psi^\mu(A)\in\scrC^\infty((T^*X\times\rz)\setminus0)$. 

\begin{definition}
$A$ is called cone-elliptic if both its homogeneous principal symbol and its rescaled principal symbol are invertible. 
\end{definition}

From now on we shall assume that $A$ is cone-elliptic. A second symbol of importance is the 
so-called conormal symbol of $A$, 
 $$\sigma_M^\mu(A)(z)=\sum_{k=0}^\mu a_k(0)z^k,$$
which is a polynomial whose coefficients are differential operators on $X$. 
The cone-ellipticity of $A$ implies that 
$\sigma_M^\mu(A)$ is meromorphically invertible, i.e.,  
$\sigma_M^\mu(A)^{-1}$ is meromorphic with values in the pseudodifferential 
operators of order $-\mu$ on $X$. Moreover, any vertical strip in the complex plane of finite width contains only 
finitely many poles of $\sigma_M^\mu(A)^{-1}$ and the Laurent coefficients are smoothing pseudodifferential operators on $X$ 
of finite rank. 

\subsection{Closed extensions}\label{sec:02.1}

The analysis of the closed extensions of cone differential operators has a long history, see for example 
the works \cite{Lesc}, \cite{GiMe}, \cite{GKM2} and \cite{SchrSe}. Here we summarize some results. 

$A$ considered as the unbounded operator \eqref{eq:intro02} is closable 
and has two canonical closed extensions: 
\begin{itemize}
 \item[i$)$] The \textit{closure}  
 $A_{\min}$, which is given by the action of $A$ on the domain 
  \begin{equation*}
    \scrD_{\min}(A)=\Big\{
    u\in\mathop{\mbox{\LARGE$\cap$}}_{\eps>0}\calH^{s+\mu,\gamma+\mu-\eps}_p(\bz)\mid 
    Au\in\calH^{s,\gamma}_p(\bz)
    \Big\}.
\end{equation*}
In case the conormal symbol of $A$ is invertible $($as a pseudodifferential operator on $X)$ for every $z$ 
with $\re z=\frac{n+1}{2}-\gamma-\mu$, the minimal domain coincides with $\calH^{s+\mu,\gamma+\mu}_p(\bz)$. 
 \item[ii$)$] The \textit{maximal extension} $A_{\max}$, given by the action of $A$ 
  on the domain 
   $$\scrD_{\max}(A)=\big\{u\in\calH^{s,\gamma}_p(\bz)\mid Au\in\calH^{s,\gamma}_p(\bz)\big\}.$$
\end{itemize}
It should be clear from the context to which particular choice of $s$, $p$ and $\gamma$ we refer.  

Similarly, we consider  the model cone operator $\widehat A$ associated with $A$ as an unbounded operator
\begin{equation}\label{eq:ext01.5}
 \wh{A}:\scrC^{\infty}_\comp(X^\wedge)\subset\calK^{s,\gamma}_p(X^\wedge)^\rho\longrightarrow 
 \calK^{s,\gamma}_p(X^\wedge)^\rho  
\end{equation}
where $\rho$ ranges over $\rz$ $($again, the choice of parameters will not be specified in the notation). 
In case $\wh{A}$ satisfies a suitable ellipticity condition on $X^\wedge$ near 
$t=+\infty$ $($which is satisfied, for example, if $A$ satisfies condition $($E$1)$ in Section \ref{sec:03}$)$, 
minimal and maximal extension, denoted by $\wh{A}_{\min}$ and 
$\wh{A}_{\max}$, respectively, are as above, substituting the Sobolev spaces $\calH^{\cdot,\cdot}_p(\bz)$ 
by $\calK^{\cdot,\cdot}_p(X^\wedge)^\rho$. 

It turns out that the gap between minimal and maximal domain is finite-dimensional.
In the following let $\omega\in\scrC_{\mathrm{comp}}^\infty([0,1))$ denote a cut-off function.

\begin{proposition}\label{lem:ext01}
There exist subspaces $\scrE,\wh{\scrE}\subset \scrC^\infty(X^\wedge)$ 
$($which do not depend on the choice of $s$, $p$ and $\rho)$, both finite-dimensional and 
of same dimension, such that 
 $$\scrD_{\max}(A)=\scrD_{\min}(A)\oplus\omega\scrE,\qquad 
     \scrD_{\max}(\wh{A})=\scrD_{\min}(\wh{A})\oplus\omega\wh{\scrE}.$$
\end{proposition}

The elements of both $\scrE$ and $\wh{\scrE}$ are finite linear combinations of functions of the form 
$c(x)t^{-p}\ln^kt$ with $c(x)\in\scrC^\infty(X)$, $p\in\cz$, and $k\in\nz_0$. For an explicit description 
see the following subsection.

Therefore, any closed extension $\ulA$ of $A$ is given by the action of $A$ on a 
domain of the form  
\begin{equation}\label{eq:ext01}
 \scrD(\ulA)=\scrD_{\min}(A)\oplus\omega\ulE,\qquad 
   \text{$\ulE$ subspace of $\scrE$}, 
\end{equation}
while any closed extension $\ulwhA$ of $\wh{A}$ is given 
by the action of $\wh{A}$ on a domain of the form  
\begin{equation}\label{eq:ext02}
 \scrD(\ulwhA)=\scrD_{\min}(\wh{A})\oplus\omega\ulwhE,\qquad 
   \text{$\ulwhE$ subspace of $\wh{\scrE}$}.  
\end{equation}

For later purpose we present the following result:

\begin{proposition}\label{prop:independent}
Assume that the conormal symbol of $A$ is invertible for every $z$ with $\re z=\frac{n+1}{2}-\gamma-\mu$. 
Consider 
 $$\ulA:\calH^{s+\mu,\gamma+\mu}_p(\bz)\oplus\omega\ulE\lra \calH^{s,\gamma}_p(\bz).$$
Then both Fredholm property and invertibility are independent on the choice of $s$ and $p$. 
Also the index does not depend on $s$ and $p$. 
\end{proposition} 

The analogous result holds true for the extensions of the model cone operator $\wh{A}$, with indepndence on the involved 
parameters $s$, $p$, and $\rho$. 

\begin{proof}
Since $\ulE$ is finite-dimensional, $\ulA$ is a Fredholm operator if and only if 
$A:\calH^{s+\mu,\gamma+\mu}_p(\bz)\to\calH^{s,\gamma}_p(\bz)$ is a Fredholm operator. The index of both operators 
is the same, due to the stability of the index under compact perturbations. 
By assumption, $A$ is an elliptic element in the algebra of cone pseudodifferential operators 
$C^\mu(\bz;(\gamma+\mu,\gamma,(-N,0]))$ for arbitrary $N$. According to Corollary 3.5 of \cite{SchrSe2001}, 
both Fredholm property and index of $A$ do not depend on the choice of $s$ and $p$.  

Now assume that $\ulA$ is invertible for some fixed choice of the parameters. We just have seen that $\ulA$ 
is a Fredholm operator of index $0$ for every choice of $s$ and $p$. Thus it suffices to show that $\ulA$ is always 
injective. So let $u+e$ with $u\in\calH^{s+\mu,\gamma+\mu}_p(\bz)$ and $e\in\omega\ulE$ belong to the kernel. Then 
$Au=-Ae\in\scrC^{\infty,\gamma}_0(\bz)$. By elliptic regularity in the cone algebra we conclude that 
$u\in\scrC^{\infty,\mu+\gamma}_0(\bz)$. This shows that the kernel of $\ulA$ does not depend on $s$ and $p$, 
and neither does the injectivity. 
\end{proof}

\subsection{One-to-one correspondence between closed extensions}\label{sec:02.2}

There exists a certain one-to-one correspondence between the closed extensions of $A$ and those of its 
model cone operator which plays a fundamental role in the theory of parameter-ellipticity of closed extensions. 

With $A$ we associate the sequence of conormal symbols
\begin{equation}\label{eq:conormal}
 f_\ell(z)=\sum_{j=0}^\mu a_j^{(\ell)}z^j,\qquad 
 a_j^{(\ell)}:=\frac{1}{\ell!}\frac{d^\ell a_j}{dt^\ell}(0);
\end{equation}
in particular, $f_0 =\sigma_M^\mu(A)$. Below, we will use the following notation: 
\begin{equation}\label{eq:sigma}
  S_\gamma=\Big\{\sigma\in\cz\mid \sigma\text{ is a pole of $f_0^{-1}$ and } 
  \frac{n+1}{2}-\gamma-\mu<\re\sigma<\frac{n+1}{2}-\gamma\Big\}.
\end{equation}
We shall identify $\scrC^\infty(X^\wedge)$ with $\scrC^\infty((0,+\infty),\scrC^\infty(X))$ and 
use the Mellin transform 
 $$\wh{u}(z)=\int_0^\infty t^{z} u(t)\frac{dt}{t},\qquad  u\in\scrC^\infty_\comp(X^\wedge).$$

The following theorem describes the space $\wh{\scrE}$ associated with the maximal extension of the 
model cone operator. 

\begin{theorem}\label{thm:max_domain1}
For $\sigma\in S_\gamma$ define 
$G_\sigma^{(0)}:\scrC^\infty_\comp(X^\wedge)\to\scrC^\infty(X^\wedge)$ by 
 $$(G_\sigma^{(0)} u)(t)=
      \int_{|z-\sigma|=\eps} t^{-z}f_0^{-1}(z)\wh{u}(z)\,\dbar z,$$
where $\eps>0$ is sufficently small.
Then 
 $$\wh\scrE=\mathop{\mbox{\Large$\oplus$}}_{\sigma\in S_\gamma} \wh\scrE_\sigma,\qquad 
     \wh\scrE_\sigma=\mathrm{range}\,G_\sigma^{(0)}.$$
\end{theorem}

The characterization of $\scrE$ is more involved. We follow here the approach of \cite{SchrSe}; 
other descriptions can be found in \cite{GiMe}, \cite{GKM}. 
  
Define recursively 
\begin{equation}\label{eq:recursion}
 g_0=1,\qquad
 g_\ell=-(T^{-\ell}f^{-1}_0)\sum_{j=0}^{\ell-1}(T^{-j}f_{\ell-j})g_j,\qquad \ell\in\nz,
\end{equation}
where the shift-operators $T^\rho$, $\rho\in\rz$, act on meromorphic functions by $(T^\rho f)(z)=f(z+\rho)$. 
The $g_\ell$ are meromorphic and the recursion is equivalent to 
\begin{equation}\label{eq:recursion2}
 \sum_{j=0}^\ell(T^{-j}f_{\ell-j}){g}_j
 =\begin{cases}f_0&\quad: \ell=0\\ 0&\quad: \ell\ge 1\end{cases}.
\end{equation}
If $h$ is a meromorphic function, denote by $\Pi_\sigma h$ the principal part of the 
Laurent series in $\sigma$; of course, if $h$ is holomorphic in $\sigma$, then $\Pi_\sigma h=0$. 

\begin{theorem}\label{thm:theta}
For $\sigma\in S_\gamma$ and $\ell\in\nz$ define  
$G_\sigma^{(\ell)}:\scrC^\infty_\comp(X^\wedge)\to\scrC^{\infty}(X^\wedge)$ by 
\begin{equation}\label{eq:sigma0}
 (G_\sigma^{(\ell)} u)(t)=
  t^\ell\,\int_{|z-\sigma|=\eps} t^{-z}g_\ell(z)\,\Pi_\sigma(f_0^{-1}\,\wh{u})(z)\,\dbar z,
\end{equation}
as well as 
\begin{equation}\label{eq:gsigma}
 G_\sigma:=\sum_{\ell=0}^{\mu_\sigma}G_\sigma^{(\ell)},\qquad
 \mu_\sigma:=\Big[\re\sigma+\mu+\gamma-\frac{n+1}{2}\Big],
\end{equation}
where $[x]$ denotes the integer part of $x\in\rz$. Then 
 $$\scrE=\mathop{\mbox{\Large$\oplus$}}_{\sigma\in S_\gamma} \scrE_\sigma,
   \qquad \scrE_\sigma=\mathrm{range}\,G_\sigma.$$
Moreover, the following map is well-defined and an isomorphism$:$
\begin{equation}\label{eq:isom1}
  \theta_\sigma:\scrE_\sigma\longrightarrow\wh{\scrE}_\sigma,\quad 
   G_\sigma(u)\mapsto G^{(0)}_\sigma(u).
\end{equation}
\end{theorem}

This is a consequence of Propositions \ref{prop:max} and \ref{prop:basis}, below. 
The maps $\theta_\sigma$ induce a one-to-one correspondence between the subspaces of 
$\scrE=\mathop{\mbox{\Large$\oplus$}}_{\sigma\in S_\gamma}\scrE_\sigma$ 
and $\wh{\scrE}=\mathop{\mbox{\Large$\oplus$}}_{\sigma\in S_\gamma}\wh{\scrE}_\sigma$, respectively, 
i.e., an  isomorphism 
\begin{equation}\label{eq:Grassmannian}
 \Theta:\mathrm{Gr}(\scrE)\longrightarrow \mathrm{Gr}(\wh{\scrE})
\end{equation}
between the corresponding Grassmannians. Hence we obtain a one-to-one correspondence between 
the closed extensions of $A$ and $\wh{A}$, respectively. 

\subsubsection{An example} \label{subsec:example}
The operators $G^{(\ell)}_\sigma$ introduced above are explicitly computable by  the residue 
theorem. 

Let us consider  a second order operator  $A$ whose inverted conormal symbol $f_0^{-1}$ has only simple 
poles. This happens, for instance, when $A$ is the conical Laplacian and $\bz$ has dimension larger or equal than $3$; in 
the two-dimensional case there is, in addition, one double pole in $z=0$ $($cf. Section \ref{sec:applications} for more details$)$. 

Let $\sigma$ be such a pole and denote by $\alpha_\sigma$ the residue of $f_0^{-1}$ in $\sigma$. Recall that 
$\alpha_\sigma$ is a smoothing pseudodifferential operator on $X$ and that $E_\sigma:=\mathrm{range}\,
\alpha_\sigma$ is a finite-dimensional subspace of $\scrC^\infty(X)$. 
Using the above notation, 
 $$G^{(0)}_\sigma(u)=t^{-\sigma}\alpha_\sigma(\wh{u}(\sigma)),\qquad u\in\scrC^\infty_{\mathrm{comp}}(X^\wedge).$$
Since the range of $u\mapsto\wh{u}(\sigma)$ is $\scrC^\infty(X)$, this implies that 
 $$\wh{\scrE}_\sigma=\big\{t^{-\sigma}e\mid e\in E_\sigma\big\}.$$

In case $\frac{n+1}{2}-\gamma-2<\re \sigma<\frac{n+1}{2}-\gamma-1$ we have $\scrE_\sigma=\wh{\scrE}_\sigma$, since then 
$G_\sigma=G_\sigma^{(0)}$. 
In case $\frac{n+1}{2}-\gamma-1\le\re \sigma<\frac{n+1}{2}-\gamma$, the structure of $\scrE_\sigma$ depends on 
$g_1=-(T^{-1}f_0^{-1})f_1$: Write, near $\sigma$, 
   $$g_1(z)\equiv \beta_\sigma(z-\sigma)^{-1}+\beta^{0}_\sigma$$
modulo a holomorphic function vanishing in $\sigma$. In case $g_1$ is holomorphic in $\sigma$, obviously $\beta_\sigma=0$ 
and  $\beta^{0}_\sigma=g_1(\sigma)$. Now one  computes 
   $$G^{(1)}_\sigma(u)=t^{-\sigma+1}\Big(\beta^0_\sigma\alpha_\sigma(\wh{u}(\sigma))
       +\beta_\sigma\alpha_\sigma(\wh{u}(\sigma))\log t\Big).$$
  It follows that 
   $$\scrE_\sigma=\big\{t^{-\sigma}e+t^{-\sigma+1}\big(\beta^0_\sigma e+\beta_\sigma e\log t\big) \mid e\in E_\sigma\big\}.$$

\subsection{The proof of Theorem \ref{thm:theta}}\label{sec:appendix}

Let $\omega, \omega_1 \in \scrC^\infty_{\mathrm{comp}}([0,1))$ be  cut-off functions. 

\begin{proposition}\label{prop:max}
$\omega\scrE_\sigma$ is a subspace of $\scrD_{\max}(A)$. 
\end{proposition}
\begin{proof}
By construction, $\omega\scrE_\sigma$ is contained in $\scrC^{\infty,\gamma}_0(\bz)$ because 
it consists of functions of the form $\omega(t)c(x)t^{-\sigma}\ln^\ell t$ with $c(x)\in\scrC^\infty(X)$ 
and $\re\sigma<(n+1)/2-\gamma$. 
 
Now let $v=\omega G_\sigma(u)$ with $u\in\scrC^\infty_{\mathrm{comp}}(X^\wedge)$. 
We show that $Av$ belongs to $\scrC^{\infty,\gamma}_0(\bz)$. Choose $\omega_1$ such that 
$\omega\equiv1$ in a neighborhood of the support of $\omega_1$. Since $(1-\omega_1)Av$ is smooth 
and compactly supported in the interior of $\bz$, it suffices to analyze $\omega_1Av$. 
Now we can write  
 $$\omega_1A=\omega_1 t^{-\mu}\sum\limits_{j=0}^{\mu-1} t^j f_j(-t\partial_t)+R,$$
with a remainder $R$ that maps $\scrC^{\infty,\gamma}_0(\bz)$ into itself. 
Observing that 
$\omega G^{(\ell)}_\sigma$ maps $\scrC^{\infty}_{\mathrm{comp}}(X^\wedge)$ into $\scrC^{\infty,\gamma+\mu-(\mu_\sigma+1)+\ell}_0(\bz)$,  
we see that $\omega_1Av$ belongs to $\scrC^{\infty,\gamma}_0(\bz)$ provided 
\begin{align}\label{eq:sum}
\begin{split}
 \omega_1\, t^{-\mu}&\sum_{j=0}^{\mu_\sigma}\sum_{\ell=0}^{\mu_\sigma-j}
     t^j f_j(-t\partial_t) G^{(\ell)}_\sigma(u)\\
  &=\omega_1\, t^{-\mu}\sum_{j=0}^{\mu_\sigma}
     \sum_{\ell=0}^{\mu_\sigma-j}t^{j+\ell}(T^{-\ell} f_j)(-t\partial_t)(t^{-\ell} G^{(\ell)}_\sigma(u))
     \;\in\; \scrC^{\infty,\gamma}_0(\bz);
\end{split}
\end{align}
note that we have used the Mellin operator identity $f(-t\partial_t)t^{\rho}=t^{\rho}(T^{-\rho}f)(-t\partial_t)$.   
Rearranging the summation \eqref{eq:sum} equals 
\begin{equation}\label{eq:av}
 \omega_1\, 
   t^{-\mu}\sum_{k=0}^{\mu_\sigma}t^k\sum_{\ell=0}^{k}(T^{-\ell}f_{k-\ell})(-t\partial_t)
   (t^{-\ell} G^{(\ell)}_\sigma(u))
   \;\in\; \scrC^{\infty,\gamma}_0(\bz). 
\end{equation}

Inserting the expression \eqref{eq:sigma0} for $G_\sigma^{(\ell)}$, the summation over $\ell$ in \eqref{eq:av} then yields, for each $k$,  
 $$\int_{|z-\sigma|=\varepsilon}
     t^{-z}\sum_{\ell=0}^{k}(T^{-\ell}f_{k-\ell})(z)g_\ell(z)
   \Pi_\sigma(f_0^{-1}\wh{u})(z)\,\dbar z.$$
Using \eqref{eq:recursion2}, we conclude that \eqref{eq:av} is equal to zero.  
\end{proof}

\begin{proposition}\label{prop:basis}
Let $u,v\in\scrC^{\infty}_{\mathrm{comp}}(X^\wedge)$. Then $G_\sigma(u)=G_\sigma(v)$ if and only if 
$G^{(0)}_\sigma(u)=G^{(0)}_\sigma(v)$. In particular, $\scrE_\sigma$ has the same 
dimension as $\wh{\scrE}_\sigma$. 
\end{proposition}

\begin{proof}
Set $w=u-v$ and write 
 $$\Pi_\sigma(f_0^{-1}\wh{w})(z)=\sum_{\ell=0}^m c_\ell(z-\sigma)^{-\ell-1}$$
with coefficient functions $c_\ell\in\scrC^\infty(X)$. Since $t^{-z} =t^{-\sigma}\exp(-(z-\sigma)\log t)$, 
$$\res_{z=\sigma}t^{-z}\Pi_\sigma(f_0^{-1}\wh{w})(z) 
 =t^{-\sigma}\sum_{\ell=0}^m \frac{(-1)^\ell}{\ell!}c_\ell\log^{\ell}t.$$
Thus $G^{(0)}_\sigma(w)=0$ if and only if all $c_\ell\equiv0$, i.e., if and only if 
$\Pi_\sigma (f_0^{-1}\wh{w})\equiv0$. This obviously implies $G_\sigma(w)=0$. 
Conversely, $G_\sigma(w)=0$ implies that 
$G^{(0)}_\sigma(w)=-\sum\limits_{\ell=1}^{\mu_\sigma}G_\sigma^{(\ell)}(w)$. 
However, by construction,  
 $$\mathrm{range}\,G^{(0)}_\sigma\cap 
   \mathrm{range}\,\sum_{\ell=1}^{\mu_\sigma}G_\sigma^{(\ell)}=\{0\}.$$
This shows $G^{(0)}_\sigma(w)=0$. 
The same argument shows that functions 
$G^{(0)}_\sigma(u_1)$,$\ldots$, $G^{(0)}_\sigma(u_r)$
are linearly independent in $\mathrm{range}\,G_\sigma^{(0)}$, if and only if
$G_\sigma(u_1),\ldots,G_\sigma(u_r)$ are linearly independent in $\mathrm{range}\,G_\sigma$.
\end{proof}

\section{Parameter-ellipticity and resolvent of closed extensions}\label{sec:03}

Let $\Lambda$ be the sector from \eqref{eq:intro01} and $\ulA$ a closed extension of $A$ in 
$\calH_p^{s,\gamma}(\bz)$ with domain 
$\scrD(\ulA)=\calH^{s+\mu,\gamma+\mu}_p(\bz)\oplus \omega\ulE$ for a subspace $\ulE$ of $\scrE$. 
We will next state  three conditions which will allow us to construct the resolvent to $\ulA$ for large $\lambda$ in 
$\Lambda$ and to determine its structure. 
Actually, these conditions are independent of $s$ and $p$. They involve the model cone operator 
$\wh{A}$, considered as unbounded operator in $\calK^{0,\gamma}_2(X^\wedge)$ with domain 
$ \scrD(\ulwhA)=\calK^{\mu,\gamma+\mu}_2(X^\wedge)\oplus \ulwhE$, where $\ulwhE:=\Theta^{-1}\ulE$, see \eqref{eq:Grassmannian}. 


We call $\ulA$  \textit{parameter-elliptic with respect to $\Lambda$}, if 
\begin{itemize}
 \item[(E1)]  Both $\lambda-\sigma^\mu_\psi(A)$ and 
 $\lambda-\wt\sigma^\mu_\psi(A)$ are invertible  in the sector $\Lambda$.
 \item[(E2)] The principal conormal symbol $\sigma^\mu_M(A)(z)$ is invertible for all $z\in\cz$ with 
  $\re z=\frac{n+1}{2}-\gamma-\mu$ or $\re z=\frac{n+1}{2}-\gamma$. 
 \item[(E3)] $\Lambda$ is a sector of minimal growth for $\ulwhA$, i.e., there exist $C,R\ge0$ such that, 
  for $\lambda \in \Lambda$, $|\lambda|\ge R$, the operator $\lambda-\ulwhA$ is invertible and 
  $$\|\lambda (\lambda-\ulwhA)^{-1}\|_{\scrL(\calK^{0,\gamma}_2(X^\wedge))}\le C.$$
 \end{itemize}

Condition $($E$2)$ assures that   $\scrD_{\min}(A)= \calH^{s+\mu,\gamma+\mu}_p(\bz)$  and $\scrD_{\min}(\wh{A}) = \calK^{\mu,\gamma+\mu}_2(X^\wedge)$,
cf. i$)$ in the beginning of Section \ref{sec:02.1}. 
The invertibility for $z$ with real part $\frac{n+1}{2}-\gamma$ is a symmetry condition
used for treating the adjoint.

Below, it will be convenient to replace the variable $\lambda$ by $\eta^\mu$ in order to 
raise the order of the parameter from $1$ to $\mu$, which is the order of $A$. 
So let  
 $$\Sigma=\Sigma(\Lambda)=\left\{se^{i\varphi}\mid s\ge 0,\;
   \theta\le\mu\varphi\le 2\pi-\theta\right\};$$ 
then $\eta\mapsto\eta^\mu:\Sigma\longrightarrow\Lambda$ is a bijective map. 

\forget{
As we shall see in Section \ref{sec:03.3}, below, condition $($E$3)$ is independent of $\rho$ in the sense that, 
if it holds for some particular value $\rho=\rho_0$, it also holds for arbitrary $\rho\in\rz$. We shall use this fact from now on. 
A similar comment concerns the choice of $p=2$.
To show this, however, requires a lengthy argument; we shall confine ourselves to  
Remark \ref{rem:p-ind}, below. 
We think that this is sufficient, since condition $($E$3)$ is most easily verified in the Hilbert space case $p=2$.
}

\subsection{Ellipticity condition $($E$3)$}\label{sec:03.3}

We shall demostrate that, in $($E$3)$, we could as well consider $\wh{A}$ as an unbounded operator 
in $\calK^{0,\gamma}_2(X^\wedge)^\rho$ with an arbitrary choice of  $\rho\in\rz$. 

In fact, as mentioned after Proposition \ref{prop:independent}, the invertibility of 
  $$\lambda-\ulwhA:\calK^{\mu,\gamma+\mu}_2(X^\wedge)^\rho\oplus\omega\ulwhE\lra 
      \calK^{0,\gamma}_2(X^\wedge)^\rho$$
is independent of the choice of $\rho$. We shall show that this is also true for the finiteness of the supremum 
\begin{equation}\label{eq:res}
  \sup_{\lambda\in\Lambda,\,|\lambda|\ge R}\|\lambda(\lambda-\ulwhA)^{-1}\|_{\scrL(\calK^{0,\gamma}_2(X^\wedge)^\rho)}.
\end{equation}

Assume that \eqref{eq:res} is finite for some $\rho=\rho_0$. 
Recall that for an unbounded operator $T:\scrD\subset X\to X$ in a Banach space $X$, the uniform boundedness of $\lambda(\lambda-T)^{-1}$ 
in $\scrL(X)$  for $\lambda $  in a $($truncated$)$ sector is equivalent to that of  $(\lambda-T)^{-1}$ in  $\scrL(X,\scrD)$, where $\scrD$ carries the  graph norm. 
As the domain of $\wh{A}$ is continuously embedded in  $\calK^{\mu,\gamma}_2(X^\wedge)^\rho$, 
it follows that $(\lambda-\ulwhA)^{-1}$ is uniformly 
bounded in $\scrL(\calK^{0,\gamma}_p(X^\wedge)^{\rho_0},\calK^{\mu,\gamma}_2(X^\wedge)^{\rho_0})$. 
The complex interpolation identity  
 $$\big(\calK^{0,\gamma}_{p}(X^\wedge)^\rho,\calK^{\mu,\gamma}_{p}(X^\wedge)^\rho\big)_{[\theta]}
     =\calK^{\mu\theta ,\gamma}_p(X^\wedge)^\rho$$
implies for $\theta = (\mu-1)/\mu$ the uniform estimate 
\begin{equation}\label{eq:interpolation}
 \|(\lambda-\ulwhA)^{-1}\|_{\scrL\left(\calK^{0,\gamma}_{2}(X^\wedge)^{\rho_0},
 \calK^{\mu-1,\gamma}_{2}(X^\wedge)^{\rho_0}\right)}
 \le C\,|\lambda|^{-1/\mu},\qquad |\lambda|\ge R.
\end{equation}
Now consider another choice $\rho=\rho_1$ and let $r$ be a smooth positive function on $\rz_+$ with $r(t)=1$ for $t\le 1$ 
and $r(t)=t^{\rho_0-\rho_1}$ for large $t$. 
Set $\wh{B}:=r^{-1}\wh{A}r$. 
Note that $\wh{A}-\wh{B}=r^{-1}[r,\wh{A}]$ vanishes on $(0,1)\times X$ and is of 
order $\mu-1$, i.e., 
\begin{equation}\label{eq:regular}
 \wh{A}-\wh{B}:\calK^{s,\gamma}_{2}(X^\wedge)^{\rho}\lra\calK^{s-\mu+1,\gamma}_{2}(X^\wedge)^{\rho}
     \qquad\forall\;s,\rho\in\rz
\end{equation}
$($actually, on the right-hand side one can replace $\rho$ by the better weight $\rho+1$; 
however, we shall not need this fact$)$. 

Since multiplication by $r$ induces isomorphisms from $\calK^{s,\gamma}_{p}(X^\wedge)^{\rho_0}$ to 
$\calK^{s,\gamma}_{2}(X^\wedge)^{\rho_1}$, studying the resolvent of $\ulwhA$ in 
$\calK^{0,\gamma}_{2}(X^\wedge)^{\rho_1}$ is  equivalent to studying the resolvent of $\ulwhB$ in 
$\calK^{0,\gamma}_{2}(X^\wedge)^{\rho_0}$, where $\ulwhB$ has the same domain  
in $\calK^{0,\gamma}_{2}(X^\wedge)^{\rho_0}$ as $\ulwhA$. 
The resolvent identity 
 $$(\lambda-\ulwhA)^{-1}-(\lambda-\ulwhB)^{-1}=(\lambda-\ulwhB)^{-1}(\ulwhA-\ulwhB)(\lambda-\ulwhA)^{-1}$$
yields 
 $$(\lambda-\ulwhB)^{-1}=(\lambda-\ulwhA)^{-1}\Big[1+(\wh{A}-\wh{B})(\lambda-\ulwhA)^{-1}\Big]^{-1}. $$
By \eqref{eq:interpolation} and \eqref{eq:regular} $($with $s=\mu-1)$ the norm of  
$(\wh{A}-\wh{B})(\lambda-\ulwhA)^{-1}$ in $\scrL(\calK^{0,\gamma}_{2}(X^\wedge)^{\rho_0})$ is $O(|\lambda|^{-1/\mu})$. 
A von Neumann series argument implies that  
the inverse $[\ldots]^{-1}$ exists and is uniformly bounded 
in $\lambda$. We deduce that $(\lambda-\ulwhB)^{-1}$ decays like $|\lambda|^{-1}$, showing that \eqref{eq:res} 
also holds for the choice $\rho=\rho_1$.

\begin{remark}\label{rem:p-ind}
\rm
In $($E$3)$ one can also substitute $p=2$ by any other choice of $p\in(1,+\infty)$. 
However, there seems to be no analog of the simple proof used above.
Instead, one needs to show that the resolvent is an element of a  calculus of 
parameter-dependent pseudodifferential operators on the infinite cone $X^\wedge$. 
Then general mapping properties of such  operators give the norm-estimate 
of the resolvent simultaneously for all $p$.  
It exceeds the scope of this paper to go into the  details. 
Anyway, condition $($E$3)$ is most easily verified in the Hilbert space case $p=2$. 
\end{remark}

\subsection{Parameter-dependent Green operators}\label{sec:03.1}

We will describe the structure of the resolvent of $\ulA$, using Schulze's calculus for
parameter-dependent operators on conical manifolds with the slight modification that the 
parameter-dependent Green operators are not assumed to be classical. 
We will next discuss this in more detail. 
 
For $\gamma,\gamma^\prime\in\rz$ and $\eps>0$ let   
\begin{align*}
 \scrS^{{\gamma^\prime},\gamma}_\eps(X^\wedge\times X^\wedge)
    &=\scrS^{\gamma^\prime}_\eps(X^\wedge)\,\wh{\otimes}_\pi\,\scrS^\gamma_\eps(X^\wedge),\\
 \scrC^{\infty,{\gamma^\prime},\gamma}_\eps(\bz\times\bz)
    &=\scrC^{\infty,{\gamma^\prime}}_\eps(\bz)\,\wh{\otimes}_\pi\,\scrC^{\infty,\gamma}_\eps(\bz). 
\end{align*}
Recall that a function $u=u(t,x)$ belongs to $\scrS^{\gamma}_\eps(X^\wedge)$, $\eps>0$, if and only if 
 $$t^{\frac{n+1}{2}-\gamma-\delta}\spk{t}^{k}(t\partial_t)^pD_xu(t,x)
     \in L^2\Big(\rz_+\times X,\frac{dt}{t}dx\Big)$$
for every choice of integers $k$ and $p$, all differential operators $D$ on $X$, and $0\le\delta<\eps$. 
\forget{
Note that a function $k(t,x,s,y)$ belongs to $\scrS^{{\gamma^\prime},\gamma}(X^\wedge\times X^\wedge)$ if, 
and only if, 
 $$t^{\frac{n+1}{2}-\gamma^\prime}s^{\frac{n+1}{2}-\gamma}
     \spk{s}^k\spk{t}^{k^\prime}\spk{y}^\ell\spk{x}^{\ell^\prime}(s\partial_s)^p(t\partial_t)^{p^\prime}A_yA^\prime_x
     k\in L^2\Big(X^\wedge\times X^\wedge,\frac{ds}{s}\frac{dt}{t}dxdy\Big)$$
for every choice of integers $k,k^\prime,\ell,\ell^\prime,p,p^\prime$ and differential operators $A,A^\prime$ on $X$. 
}

\begin{definition}
Let $\eta\mapsto [\eta]$ be a smooth positive function with $[\eta]=|\eta|$ for $|\eta|\ge 1$. 
By $\mathcal{R}^\nu_G(\Sigma;\gamma,{\gamma^\prime})$ denote the space of all operator-families $a(\eta)$, $\eta\in \Sigma$, of the form 
 $$(a(\eta)u)(t,x)=[\eta]^{n+1}\int_0^\infty\int_{X}k_a(\eta,t[\eta],x,s[\eta],y) u(s,y)\,s^n dsdy$$
with an integral kernel satisfying, for some $\eps=\eps(a)>0$,
 $$k_a(\eta,t,x,s,y)\in S^{\nu}(\Sigma,\scrS^{{\gamma^\prime},-\gamma}_\eps
    (X^\wedge_{(t,x)}\times X^\wedge_{(s,y)})).$$
\end{definition}

Here we do not require $k_g$ to be a classical symbol. 

\begin{definition}
The space $\calC^\nu_G(\Sigma;\gamma,{\gamma^\prime})$ consists of all operator-families $g(\eta)$, $\eta\in \Sigma$, of the form 
 $$g(\eta)=\omega_1\,a(\eta)\,\omega_0+r(\eta),$$
where $\omega_0,\omega_1\in\scrC^\infty([0,1))$ are cut-off functions, 
$a\in \mathcal{R}^\nu_G(\Sigma;\gamma,\rho)$, and 
 $$r\in \scrS(\Sigma,\scrC^{\infty,{\gamma^\prime},-\gamma}_\eps(\bz\times\bz))$$
for some $\eps=\eps(g)>0$. 
\end{definition}

In the representation of $g$ above, the cut-off functions can be changed at the 
cost of substituting $r$ by another element of the same structure. 

\subsubsection{A characterization of Green operators}\label{sec:03.1.1}

We shall show that parameter-depen\-dent Green operators can be characterized 
by certain mapping properties, without reference to the structure of the integral kernels. 
This characterization will be important in the proof of our main theorem. 

\begin{lemma}\label{lem:green}
Let $\mathcal{R}_G(\gamma,\gamma^\prime)_\eps$ denote the 
Fr\'echet  space of all bounded operators 
 $$A:\calK^{0,\gamma}_2(X^\wedge)\lra\calK^{0,\gamma^\prime}_2(X^\wedge)$$ 
such that 
the range of $A$ is contained in $\scrS^{\gamma^\prime}_\eps(X^\wedge)$ and the range of $A^*$  is contained in $\scrS^{-\gamma}_\eps(X^\wedge)$. 
Here the adjoint $*$ refers to the pairings induced by the inner product of $\calK^{0,0}_2(X^\wedge)$.
\forget{
\footnote{Note that $\mathcal{R}_G(\gamma,\gamma^\prime)_\eps$ is a Fr\'{e}chet space in the natural way.
 i.e., as the projective limit with respect to the maps 
\begin{align*}
 A\mapsto A&:\mathcal{R}_G(\gamma,\gamma^\prime)_\eps\to \scrL(\calK^{0,\gamma}_2(X^\wedge),\calK^{s,\gamma^\prime+\eps-1/n}_2(X^\wedge)^{\rho}),\\ 
 A\mapsto A^*&:\mathcal{R}_G(\gamma,\gamma^\prime)_\eps\to
\scrL(\calK^{0,-{\gamma^\prime}}_{2}(X^\wedge),
   \calK^{s,-\gamma+\eps-1/n}_{2}(X^\wedge)^\rho),\qquad s,\rho,n\in\nz.
\end{align*} 
}.
}
Then every such operator $A$ is an integral 
operator with kernel $k_A\in\scrS^{{\gamma^\prime},\gamma}_{\eps/2}(X^\wedge\times X^\wedge)$ 
$($with respect to the measure $t^ndtdx)$ and the following map 
is continuous$:$
 $$A\mapsto k_A:\mathcal{R}_G(\gamma,\gamma^\prime)_\eps\lra
     \scrS^{{\gamma^\prime},-\gamma}_{\eps/2}(X^\wedge\times X^\wedge).$$
\end{lemma}
\begin{proof}
Without loss of generality we may assume $\gamma=\gamma^\prime=0$. 
It suffices to show the existence of the 
kernel $k_A$ for any given $A$; the continuity of $A\mapsto k_A$ then follows from the closed graph theorem. 

Let $\phi_1,\ldots,\phi_N$ be a partion of unity on $X$. Considering these functions as constant in the variable $t$ 
we get a partition of unity of $X^\wedge$. Writing $A=\sum_{j,k}\phi_jA\phi_k$ it suffices to prove the following 
local version of the lemma$:$ Let $H=L_2(\rz_+\times\rz^n,t^ndtdx)$ and let $\mathcal{R}_\eps$ denote the space of all 
operators $A\in \scrL(H)$ such that the range of both $A$ and $A^*$ is contained in 
$\scrS^0_\eps(\rz_+\times\rz^n):=\scrS^0_\eps(\rz_+)\wh{\otimes}_\pi\scrS(\rz^n)$. Then $A$ has a kernel 
\begin{equation}\label{eq:ker01}
 k_A\in \scrS^0_{\eps/2}(\rz_+\times\rz^n)\,\wh{\otimes}_\pi\,\scrS^0_{\eps/2}(\rz_+\times\rz^n).
\end{equation}
Indeed, 
the  mapping properties and general results on tensor product representations, 
see \cite[Proposition 4.2.9]{Krai} 
imply that $A$ has a kernel  
\begin{equation}\label{eq:ker02}
 k_A\in\scrS^0_\eps(\rz_+\times\rz^n)\,\wh{\otimes}_\pi\,H\,\cap\,H\,\wh{\otimes}_\pi\,\scrS^0_\eps(\rz_+\times\rz^n).
\end{equation}
We will show that this space embeds into that in \eqref{eq:ker01}. To this end let us introduce the following 
notation$:$ For arbitrarily chosen integers $k,k^\prime,\ell,\ell^\prime,p,p^\prime, q, q^\prime$ and arbitrary $0\le\delta<\eps/2$ 
set 
 $$m_1(t)=t^{-\delta}\spk{t}^k, \quad m_2(s)=s^{-\delta}\spk{s}^{k^\prime},\quad 
     m_3(x)=\spk{x}^\ell,\quad m_4(y)=\spk{y}^{\ell^\prime}$$
and 
 $$D_1=(t\partial_t)^p,\quad D_2=(s\partial_s)^{p^\prime},\quad 
     D_3=\spk{D_x}^q,\quad D_4=\spk{D_y}^{q^\prime}.$$
We then have to show that $k_A = k_A(t,x,s,y)$ satisfies
\begin{equation}\label{eq:L2}
 m_1m_2m_3m_4D_1D_2D_3D_4 k_A\in L^2\Big(\rz_+\times \rz^n\times\rz_+\times \rz^n,s^nds\,t^ndt\,dx\,dy\Big).
\end{equation}
Let $\|\cdot\|$ denote the norm of this $L^2$-space and let $\lambda_1+\lambda_2+\lambda_3+\lambda_4=1$ 
with positive numbers $\lambda_i$. Then, by 
the inequality of arithmetic and geometric means and the triangle inequality,  
 $$\|m_1m_2m_3m_4D_1D_2D_3D_4 k_A\|\le \sum_{i=1}^4
     \|m_i^{1/\lambda_i}D_1D_2D_3D_4 k_A\|.$$
We check that each of the summands is finite. Let us consider the summand for $i=1$; the others are treated analogously. 
Recall that the Fourier 
transform induces an isometric isomorphism in $L^2(\rz^n)$, while the Mellin transform gives an isometric isomorphism from 
$L^2(\rz_+,t^ndt)$ to $L^2(\Gamma_{(n+1)/2})$, where $\Gamma_\gamma=\gamma+i\rz\cong \rz$ is a vertical line in 
the complex plane. Moreover, recall 
that $\spk{D_x}$ under the Fourier transform becomes multiplication by $\spk{\xi}$, while $t\partial_t$ under the 
Mellin transform becomes multiplication by $-z$. Therefore, with $r_1+r_2+r_3+r_4=1$, we can estimate 
\begin{align*}
 \|m_1^{1/\lambda_1}D_1D_2D_3D_4 k_A\|
  &\le \|m_1^{1/r_1\lambda_1}D_1k_A\|+ \sum_{i=2}^4\|D_1 D_i^{1/r_i}k_A\|\\
 &\le \|m_1^{1/r_1\lambda_1}D_1k_A\|+ \sum_{i=2}^4\Big(\|D_1^2k_A\|_{L^2}+\|D_i^{2/r_i}k_A\|\Big).
\end{align*}
Since $\delta<\eps/2$, we can choose  $\lambda_1\in(0,1/2)$ and $r_1\in(0,1)$ 
$\delta/r_1\lambda_1<\eps$.  
Then all terms on the right-hand side of the latter inequality are finite due to \eqref{eq:ker02}.  
\end{proof}

In the following proposition we shall employ operator-valued symbols as introduced in Section \ref{sec:twisted}. 
 
\forget{
\begin{proposition}\label{prop:green}
$a(\eta)$ belongs to $\mathcal{R}^\nu_G(\Sigma;\gamma,{\gamma^\prime})$ if and only if there exists an $\eps>0$ 
such that 
\begin{align*}
 a(\eta)\in & \mathop{\mbox{\Large$\cap$}}_{s,s^\prime,\rho,\rho^\prime\in\rz}
 S^\nu(\Sigma;\calK^{s,\gamma}_2(X^\wedge)^\rho,\calK^{s^\prime,\gamma^\prime+\eps}_2(X^\wedge)^{\rho^\prime}),\\
 a(\eta)^*\in & \mathop{\mbox{\Large$\cap$}}_{s,s^\prime,\rho,\rho^\prime\in\rz}
 S^\nu(\Sigma;\calK^{s^\prime,-{\gamma^\prime}}_{2}(X^\wedge)^{\rho^\prime},
   \calK^{s,-\gamma+\eps}_{2}(X^\wedge)^\rho),
\end{align*}
where the $($pointwise$)$ adjoint refers to the pairing$($s$)$ induced by the inner product of $\calK^{0,0}_2(X^\wedge)$. 
\end{proposition}
}
\begin{proposition}\label{prop:green} We have 
$a\in\mathcal{R}^\nu_G(\Sigma;\gamma,{\gamma^\prime})$ if, and only if, there exists an $\eps>0$ 
such that 
\begin{align*}
 a\in S^\nu(\Sigma;\calK^{0,\gamma}_2(X^\wedge),\scrS^{\gamma^\prime}_\eps(X^\wedge)),\qquad
 a^*\in   S^\nu(\Sigma;\calK^{0,-{\gamma^\prime}}_{2}(X^\wedge), \scrS^{-\gamma}_\eps(X^\wedge)),
\end{align*}
where the pointwise adjoint refers to the pairings induced by the inner product of $\calK^{0,0}_2(X^\wedge)$. 
\end{proposition}
\begin{proof}
It is easy to show that every $a\in\mathcal{R}^\nu_G(\Sigma;\gamma,{\gamma^\prime})$ has the stated properties. 

Thus let us assume that $a$ and $a^*$ are as described with some $\eps>0$. 
Due to Lemma \ref{lem:green}, $a(\eta)$ has an integral kernel 
$k(\eta)=k(\eta;t,x,s,y)$ belonging to 
$\scrC^\infty\big(\rz^q,\scrS^{\gamma,\gamma^\prime}_{\eps/2}(X^\wedge\times X^\wedge)\big)$. 
The same is then true for the kernel 
\begin{align*}
 \wt{k}(\eta;t,x,s,y)=[\eta]^{-(n+1)}k(\eta;t[\eta]^{-1},x,s[\eta]^{-1},y),
\end{align*}
of $\kappa^{-1}(\eta)a(\eta)\kappa(\eta)$. We will verify that 
$\wt{k}$ belongs to $S^\nu\big(\rz^q;\scrS^{\gamma,\gamma^\prime}_{\eps/2}(X^\wedge\times X^\wedge)\big)$.

Consider $\wt{a}_\beta(\eta):=\kappa^{-1}(\eta)D^{\beta}a(\eta)\kappa(\eta)$ for arbitrary $\beta$. 
Since $\spk{\eta}^{|\beta|-\nu}\wt{a}_\beta(\eta)$ is uniformly bounded in $\mathcal{R}_G(\gamma,\gamma^\prime)_\eps$, 
it follows from Lemma \ref{lem:green} that the associated kernel 
$\spk{\eta}^{|\beta|-\nu}\wt{k}_\beta(\eta)$ is uniformly bounded in 
$\scrS^{\gamma,\gamma^\prime}_{\eps/2}(X^\wedge\times X^\wedge)$.
Since the kernel $k_\beta(\eta)$ of 
$D^\beta a(\eta)$ is 
\begin{align*}
 k_\beta(\eta;t,x,s,y)=[\eta]^{n+1}\wt{k}_\beta(\eta;t[\eta],x,s[\eta],y), 
\end{align*}
a straightforward calculation shows that 
 $$\wt{k}_{\beta+e_j}(\eta)=D_{\eta_j}\wt{k}_\beta(\eta)+\frac{D_{\eta_j}[\eta]}{[\eta]}
    \big((n+1)+(t\partial_t)+(s\partial_s)\big)\wt{k}_\beta(\eta).$$
Thus, by induction, 
$D^\alpha\wt{k}(\eta)$ is a finite linear combination of terms of the form 
 $$p_{m}(\eta)\big((n+1)+(t\partial_t)+(s\partial_s)\big)^\ell\wt{k}_\beta(\eta),\qquad \beta\le\alpha,
     \quad m+|\beta|=|\alpha|,
     \quad\ell\in\nz,$$
with symbols $p_m(\eta)\in S^{-m}(\rz^q)$. Since $(n+1)+(t\partial_t)+(s\partial_s)$ is a continuous operator 
in $\scrS^{\gamma,\gamma^\prime}_{\eps/2}(X^\wedge\times X^\wedge)$, it follows that 
$\spk{\eta}^{|\alpha|-\nu}D^\alpha\wt{k}(\eta)$
is unifomly bounded in $\scrS^{\gamma,\gamma^\prime}_{\eps/2}(X^\wedge\times X^\wedge)$. 
\end{proof}

\subsection{The resolvent construction}\label{sec:03.2}

We shall prove the following theorem:

\begin{theorem}\label{thm:resolvent}
Let $\ulA$ be parameter-elliptic with respect to $\Lambda$, i.e., satisfy conditions $\mathrm{(E1)}$, $\mathrm{(E2)}$ 
and $\mathrm{(E3)}$. 
Then there exists a constant $c\ge 0$ such that 
\begin{equation}\label{eq:schrohe1}
 \ulA+c: \calH^{s+\mu,\gamma+\mu}_p(\bz)\oplus\omega\ulE \lra \calH^{s,\gamma}_p(\bz)
\end{equation}
 has no spectrum in $\Lambda$. Moreover, we then have 
\begin{equation}\label{eq:schrohe2}
 (\eta^\mu-({\ulA}+c))^{-1}\in C^{-\mu}_O(\Sigma)+\calC^{-\mu}_G(\Sigma;\gamma,\gamma),\qquad 
\eta\in\Sigma.
\end{equation}
\end{theorem}

It is sufficient to consider the case $s=0$ and $p=2\colon$ 
According to Proposition \ref{prop:independent} the invertibility of \eqref{eq:schrohe1} does not 
depend on $s$ and $p$. Assume that 
$ (\eta^\mu-({\ulA}+c))^{-1}=B(\eta)+G(\eta)$ as in \eqref{eq:schrohe2} for $s=0$ and $p=2$. 
For fixed $\eta$ it is shown in Theorem 3.4 of \cite{SchrSe} that the inverse can be written as $B^\prime+G^\prime$, 
where  $B^\prime$ extends continuosly to mappings $\calH^{s,\gamma}_p(\bz)\to \calH^{s+\mu,\gamma+\mu}_p(\bz)$ 
for every $p$, and $G^\prime$ extends to maps $\calH^{s,\gamma}_p(\bz)\to\omega\ulE$. It follows that 
$B(\eta)+G(\eta)=B'+G'$ induces maps $\calH^{s,\gamma}_p(\bz)\to\calH^{s+\mu,\gamma+\mu}_p(\bz)\oplus\omega\ulE$ 
for every $s$ and $p$. 
Since $\scrC^{\infty}_0(\bz)$ is dense in $\calH^{s,\gamma}_p(\bz)$ and $\calH^{s+\mu,\gamma+\mu}_p(\bz)$ 
we see that $B(\eta)+G(\eta)$ induces the inverse for arbitrary $s$ and $p$.

Next let us justify that it is also enough to verify the above theorem in  case  $\gamma=0$. 
In fact, assume that $A$ satisfies the ellipticity conditions for some $\gamma$. 
Let $t$ denote simultaneously a boundary defining function for $\bz$ and the variable $t\in\rz$. 
Define $A_\gamma=t^{-\gamma}At^\gamma$ with domain $\scrD(\ulA_\gamma)=t^{-\gamma}\scrD(\ulA)$. 
We argue that $\ulA_\gamma$ satisfies the ellipticity conditions for the weight $\gamma=0$. 
Conditions $($E$1)$ and $($E$2)$ are easily verified.  For $($E$3)$ observe that
 $$(\lambda-\ulwhA_\gamma)^{-1}=t^{-\gamma}(\lambda-\ulwhA)^{-1}t^{\gamma}.$$
So the estimate of $(\lambda-\ulwhA_\gamma)^{-1}$ in $\calK^{0,0}_2(X^\wedge)^\rho$ is equivalent to that 
of $(\lambda-\ulwhA)^{-1}$  in $\calK^{0,0}_2(X^\wedge)^{\rho+\gamma}$. Since we have shown that condition 
$($E$3)$ for $A$ is satisfied not only for $\rho=0$ but for every choice of $\rho\in\rz$, this is also true for $A_\gamma$. 
Hence $\ulA_\gamma$ satisfies the assumptions of Theorem \ref{thm:resolvent} for the weight $0$. 
Provided the theorem is true in this case, we find that 
 $$(\lambda-\ulA)^{-1}=t^{\gamma}(\lambda-\ulA_\gamma)^{-1}t^{-\gamma}$$
has the structure stated in the theorem for the weight $\gamma$.

\begin{proof}[Proof of Theorem \textnormal{\ref{thm:resolvent}} in case $p=2$ and $s=\gamma=0$]
For convenience of notation we set $A(\eta)=\eta^\mu-A$, $\ulA(\eta)=\eta^\mu-\ulA$ and similarly for 
the model cone operator. 

By Theorem 6.9 of \cite{GKM2} we know that $\ulA(\eta)^{-1}$ exists for  
$\eta\in\Sigma$ of sufficiently large modulus $|\eta|\ge c$ and that 
 $$\|\ulA(\eta)^{-1}\|_{\scrL(\calH^{0,0}_2(\bz))}\le C\spk{\eta}^{-\mu},\qquad |\eta|\ge c.$$
Hence the resolvent of ${\ulA}_c=\ulA+c$ exists for all $\eta\in\Sigma$ and satisfies the norm estimate in the whole 
sector $\Sigma$. Now we may assume without loss of generality that $c=0$, otherwise we rename 
$A_c$ by $A$ again. 

Consider $A(\eta)$ as an element of $C^\mu(\Sigma;\mu,0,k)$, the space of parameter-dependent cone 
pseudodifferential operators of order $\mu$, with an arbitrary fixed integer $0<k\le \mu$. The ellipticity assumptions 
(E1) and (E2) allow us to construct a parametrix 
$B_1(\eta)$ modulo Green operators of order $0$, i.e., 
 $$\Pi_1(\eta):=1-A(\eta)B_1(\eta)\in C^0_G(\Sigma;0,0,k).$$
$B_1(\eta)$ belongs to $C^{-\mu}(\Sigma;0,\mu,k)$. In particular, $B_1(\eta)$ also belongs to 
$C^{-\mu}_O(\Sigma)+\calC^{-\mu}_G(\Sigma;0,0)$. 
Moreover, 
$B_1(\eta)$ maps $\calH^{0,0}_2(\bz)$ into $\calH^{\mu,\mu}_2(\bz)=\scrD(A_{\min})\subset\scrD(\ulA)$. 
Hence  $A(\eta)B_1(\eta)=\ulA(\eta)B_1(\eta)$ on $\calH^{0,0}_2(\bz)$ and we can write 
\begin{equation}\label{eq:inv1}
 \ulA(\eta)^{-1}=B_1(\eta)+\ulA(\eta)^{-1}\Pi_1(\eta),\qquad \eta\in\Sigma. 
\end{equation}

Denote by $A^t$ the formal adjoint of $A$. The adjoint ${\ulA}^*$ of $\ulA$ coincides with some closed 
extension of $A^t$ which we denote by $\underline{A^t}$. Obviously 
 $$(\eta^\mu-\underline{A^t})^{-1}=[(\overline{\eta}^\mu-\ulA)^{-1}]^*
     =[\ulA(\overline{\eta})^{-1}]^*,\qquad \eta\in\overline{\Sigma}.$$
The ellipticity condition (E1) remains true for $A^t$. The conormal symbol of $A^t$ is given by 
$f_0(n+1-\mu-\overline{z})^t$, where the formal adjoint refers to the inner-product of $L^2(X)$. 
Hence $A^t$ satisfies (E2) with $\gamma=0$. As above, we thus find a parametrix 
$B_2(\eta)\in C^{-\mu}(\overline{\Sigma};0,\mu,k)$ such that 
 $$\Pi_2(\eta):=1-(\overline{\eta}^\mu-A^t)B_2(\eta)\in C^0_G(\overline{\Sigma};0,0,k)$$
and write 
 $$(\overline{\eta}^\mu-\underline{A^t})^{-1}=B_2(\eta)
     +(\overline{\eta}^\mu-\underline{A^t})^{-1}\Pi_2(\eta).$$
By passing to the adjoint we thus obtain 
\begin{equation*}
 \ulA(\eta)^{-1}=B_2(\overline{\eta})^*+\Pi_2(\overline{\eta})^*\ulA(\eta)^{-1},\qquad \eta\in\Sigma. 
\end{equation*}
Inserting this expression on the right-hand side of \eqref{eq:inv1} results in the formula 
\begin{equation}\label{eq:inv2}
 \ulA(\eta)^{-1}=B_1(\eta)+B_2(\overline{\eta})^*\Pi_1(\eta)+\Pi_2(\overline{\eta})^*\ulA(\eta)^{-1}\Pi_1(\eta),\qquad \eta\in\Sigma. 
\end{equation}
Since $B_2(\overline{\eta})^*\in C^{-\mu}(\Sigma;-\mu,0,k)\subset C^{-\mu}(\Sigma;0,0,k)$ and 
Green operators form an ideal in the parameter-dependent cone algebra, 
$B_2(\overline{\eta})^*\Pi_1(\eta)\in C^{-\mu}_G(\Sigma;0,0,k)\subset \calC^{-\mu}_G(\Sigma;0,0)$. It remains to verify that 
\begin{equation}\label{eq:inv3}
 g(\eta):=\Pi_2(\overline{\eta})^*\ulA(\eta)^{-1}\Pi_1(\eta)\in \calC^{-\mu}_G(\Sigma;0,0).
\end{equation}
To this end, first observe that 
 $$\big\|D^\alpha_\eta \ulA(\eta)^{-1}\big\|_{\scrL(\calH^{0,0}_2(\bz))}\le 
     C_\alpha\spk{\eta}^{-\mu-|\alpha|},\qquad\alpha\in\nz_0^2;$$
in fact, this follows from the above resolvent estimate and the fact that 
$D^\alpha_\eta \ulA(\eta)^{-1}$, $|\alpha|\ge1$, is a finite-linear combination of terms of the form 
$p_{k,\ell}(\eta)\ulA(\eta)^{-1-\ell}$ with polynomials $p_{k,\ell}$ of degree at most $(\mu-1)\ell-k$ and 
$k+\ell=|\alpha|$. 
It follows easily that $\omega\,\ulA(\eta)^{-1}\,\omega$ belongs to 
$S^{-\mu}(\Sigma;\calK^{0,0}_2(X^\wedge),\calK^{0,0}_2(X^\wedge))$ for every 
cut-off function $\omega\in\scrC^\infty([0,1))$ $($note that the group action $\kappa_\lambda$ is 
unitary on $\calK^{0,0}_2(X^\wedge))$. 

Both $\Pi_1(\eta)$ and $\Pi_2(\overline{\eta})^*$ belong to $\calC^{0}_G(\Sigma;0,0)$. 
Hence, composing these 
operator-families with the operator of multiplication by $1-\omega^2$ $($understood as a 
function on $\bz$, supported away from the boundary$)$, both from the left or the right, yields 
functions belonging to $\scrS(\Sigma,\scrC^{\infty,0,0}_{\eps}(\bz\times\bz))$ for some $\eps>0$. 
It follows that $g(\eta)$ differs by such an error term from $\omega\,a(\eta)\,\omega$, where 
 $$a(\eta):=\omega\,\Pi_2(\overline{\eta})^*\,\omega^2\,\ulA(\eta)^{-1}\,\omega^2\,\Pi_1(\eta)\,\omega.$$
Both $\omega\,\Pi_1(\eta)\,\omega$ and $\omega\,\Pi_2(\overline{\eta})^*\,\omega$, considered as a family of operators on 
$X^\wedge$, belong to $\mathcal{R}^0_G(\Sigma;0,0)$. 
It follows that $a(\eta)\in S^{-\mu}(\Sigma;\calK^{0,0}_2(X^\wedge),\scrS^{0}_\eps(X^\wedge))$ 
for some $\eps>0$. Arguing in the same way for $a(\eta)^*$, we conclude from Proposition \ref{prop:green} 
that $a(\eta)\in\mathcal{R}^{-\mu}_G(\Sigma;0,0)$ and thus obtain \eqref{eq:inv3}. 
\end{proof}

\section{Resolvent estimates and bounded $H_\infty$-calculus}\label{sec:04}
We continue considering the extension $\underline A$  in $\calH^{s,\gamma}_p(\bz)$ with domain 
$$\scrD(\underline A) = 
\calH^{s+\mu,\gamma+\mu}_p(\bz)\oplus \omega\ulE$$
for a fixed space $\ulE$.

\begin{theorem}\label{thm:normest}
Let $\ulA$ be  parameter-elliptic. Then, for every $s\ge0$ and $1<p<+\infty$,   
 $$\|(\eta^\mu-{\ulA})^{-1}\|_{\scrL(\calH^{s,\gamma}_p(\bz))}\le C_{s,p}\spk{\eta}^{-\mu},
      \qquad  \eta\in\Sigma,\;|\eta|\ge c,$$
with suitable constants $c$ and $C_{s,p}$; $c$ is independent of $s$ and $p$.  
\end{theorem}
\begin{proof}
For $s=0$ this is an immediate consequence of the fact that the inverse has the structure \eqref{eq:schrohe2}. 
The general case is obtained by arguing as in Step 2 of the proof of Theorem 3.3 of \cite{RoidosSchr}.  
The fact that the spectrum does not depend on $s$ and $p$ implies that neither does $c$. 
\end{proof}

\begin{theorem}\label{thm:hinfty}
Let $\underline A$ and $c$ be as in Theorem $\ref{thm:normest}$. Then $A+c$ has a bounded $H_\infty$-calculus on 
$\calH^{s,\gamma}_p(\bz)$ for every $s\ge0$ and $1<p<+\infty$. 
\end{theorem}

\begin{proof}
According to \eqref{eq:schrohe2}, we have 
$(\eta^\mu-(\underline A+c))^{-1} \in C_O^{-\mu}(\Sigma) + \calC^{-\mu}_G(\Sigma, \gamma, \gamma).$
This form of the resolvent differs from the structure used in Theorem 1 of \cite{CSS3} only by the fact that, writing  $\lambda = \eta^\mu$, 
the operator family $G(\lambda)$ used there is no longer required to be classical in $\lambda$. 
 It was already pointed out in the proof of \cite[Proposition 2]{CSS3} that this property is not necessary for the subsequent argument. 
So it follows from Theorem 2 of \cite{CSS3} that $\underline A$ has bounded imaginary powers on $\calH^{0,\gamma}_p(\bz)$.  
As shown in Step 3 of the proof the Theorem 3.3 in \cite{RoidosSchr}, 
the proof can be modified to give also the boundedness of the imaginary powers on $\calH^{s,\gamma}_p(\bz)$, $s>0$. 
Moreover, it was shown in \cite{CSS1} that a structure of the resolvent as in 
\cite[Theorem 1]{CSS3} implies the existence of a bounded 
$H_\infty$-calculus on $\calH^{0,\gamma}_p(\bz)$, see \cite[Theorem 4.1]{CSS1} based on the representation \cite[(3.11)]{CSS1}. 
Arguments analogous to those used in Step 3 of the proof of \cite[Theorem 3.3]{RoidosSchr} then imply the existence of a bounded 
$H_\infty$-calculus on $\calH^{s,\gamma}_p(\bz)$ for $s>0$.   
\end{proof}

\section{Laplacian and porous medium equation}\label{sec:applications}

\subsection{Abstract quasilinear parabolic equations}
Consider an abstract quasilinear parabolic problem of the form
\begin{gather}\label{QL}
u'(t)+A(u(t))u(t)=f(t,u(t))+g(t), \quad t\in(0,T_{0});\quad
u(0)=u_{0}	
\end{gather}
in $L^{q}(0,T_{0};X_{0})$, $1<q<\infty$, where $A(u(t)$ is, for each $t$, a closed, densely defined operator in the Banach space $X_0$ with domain $\scrD(A(u(t)))=X_{1}$, independent of $t$.

The following theorem by Cl\'{e}ment and Li, \cite[Theorem 2.1]{CL} provides a simple criterion for the existence of short time solutions: 
\begin{theorem}\label{CL} 
Assume that  there exists an open neighborhood $U$ of $u_0$ in the real interpolation space  $X_{1-1/q,q}= (X_0,X_1)_{1-1/q,q}$ such that $A(u_0): X_{1}\rightarrow X_{0}$ has maximal $L^{q}$-regularity and that  
\begin{itemize}
\item[(H1)] $A\in C^{1-}(U, \scrL(X_1,X_0))$,    
\item[(H2)] $f\in C^{1-,1-}([0,T_0]\times U, X_0)$,
\item[(H3)] $g\in L^q(0,T_0; X_0)$.
\end{itemize}
Then there exists a $T>0$ and a unique $u\in L^q(0,T;X_1)\cap W^1_q(0,T;X_0)$ solving the equation \eqref{QL} on $(0,T)$.
In particular, $u\in  C([0,T];X_{1-1/q,q})$ by interpolation.
\end{theorem}

A central property is the maximal regularity of the operator $A(u_0)$. Without going into details,  we recall the following facts, which hold in UMD Banach spaces: 
\begin{proposition}\label{prop:MR} 
{\rm (a)} The existence of a bounded $H_\infty$-calculus implies the $R$-sectoriality for the same sector according to Cl\'ement and Pr\"uss,  \cite[Theorem 4]{CP}.

{\rm (b)} Every operator, which is $R$-sectorial on $\Lambda(\theta)$ for $\theta<\pi/2$, has maximal $L^q$-regularity, $1<q<\infty$, 
see Weis \cite[Theorem 4.2]{W}.
\end{proposition}

All the Mellin-Sobolev spaces used here are UMD Banach spaces, hence the existence of a bounded $H_\infty$-calculus on $\Lambda(\theta)$ for $\theta<\pi/2$ implies maximal $L^q$-regularity.  

\subsection{The Laplacian on warped cones}\label{sec:warped}
Let $g$ be a Riemannian metric on the interior of $\bz$ that is degenerate of the form 
$$g = dt^2 + t^2 h(t)$$
on the collar neighborhood $[0,1)\times X$. Here $[0,1)\ni t\mapsto h(t)$ is a smooth family of (non-degenerate) Riemannian metrics on $X$. 
We denote by $\Delta$ the Laplace-Beltrami operator associated with $g$. 
This operator has been analyzed in detail in \cite{SchrSe} for the case of a straight conical singularity, i.e., when $h$ is constant in $t$. 
We will now extend the analysis to the more general situation.  

Let us recall a few basic facts, referring to \cite[Section 5]{SchrSe} for  more details: 
The Laplacian is a second order conically degenerate differential operator on $\bz$. 
In the  collar neighborhood it is of the form   
\begin{eqnarray}\label{eq:Laplace}
\Delta = t^{-2}\left((t\partial_t)^2 - (n-1+H(t))(-t\partial_t)+ \Delta_t  \right),
\end{eqnarray}
where $\Delta_t$ is the Laplace-Beltrami operator on the cross-section $X$ 
associated with the Riemannian metric $h(t)$, $n=\dim X$, and 
$2H(t) = t\partial_t (\log \det h(t))$. 
The conormal symbol of $\Delta$ is  
\begin{equation}\label{eq:f0}
 f_0(z)=\sigma_M(\Delta)(z) = z^2-(n-1)z + \Delta_{0}
\end{equation}
and, in the notation of \eqref{eq:conormal}, 
\begin{equation}\label{eq:f1}
 f_1(z) = \dot{\Delta}_{0} - \dot H(0)z.
\end{equation}
The model cone operator $\widehat \Delta$ is 
\begin{eqnarray}\label{eq:Laplace_model}
\widehat\Delta = t^{-2}\left((t\partial_t)^2 - (n-1+H(0))(-t\partial_t)+ \Delta_0  \right).
\end{eqnarray}

Denote by $0=\lambda_0 >\lambda_1>\ldots $ the different eigenvalues of $\Delta_{0}$ 
and by $E_0, E_1,\ldots$ the associated eigenspaces.
The non-bijectivity points  of $\sigma_M(\Delta)(z)$ are the points $z=q_j^+$ and $q_j^-$, where 
$$q_j^\pm = \frac{n-1}2\pm \sqrt{\Big(\frac{n-1)}2\Big)^2-\lambda_j},\qquad j=0,1,2,\ldots.$$
In fact, we have 
\begin{equation}\label{eq:f0inv}
 f_0(z)^{-1}=(z^2-(n-1)z+\Delta_{0})^{-1}  = \sum_{j=0}^\infty \frac{\pi_j}{(z-q^+_j)(z-q^-_j)}
\end{equation}
where $\pi_j$ is the orthogonal projection, in $L^2(X)$, onto the eigenspace $E_j$.  
The poles are always simple, except for the double pole in $z=q_0^\pm=0$ in case $n=1$. 

\subsection{Extensions with (E1), (E2), and (E3)}
Considering $\Delta$ as an unbounded operator in $\calH^{s,\gamma}_p(\bz)$ for 
$s,\gamma\in \rz$, $1<p<\infty$, we are interested in the question 
which of its closed extensions satisfy the 
assumptions (E1), (E2) and (E3) and therefore have a bounded $H_\infty$-calculus. 
As the sector $\Lambda$ we choose  $\Lambda(\theta)$ for arbitrary $\theta>0$. 
Clearly, (E1) is always fulfilled. (E2) will hold for every $\gamma$ satisfying 
\begin{eqnarray}\label{eq:gamma}
\frac{n+1}2-\gamma,\,\frac{n+1}2-\gamma-2\notin \{q_j^\pm \mid j=0,1,\ldots\}.
\end{eqnarray}
We will assume this in the sequel. It remains to check (E3). 

With $q_j^\pm$, $j=1,2,\ldots$,  we associate the function space 
$$\wh{\scrE}_{q_j^\pm} =t^{-q_j^\pm}\otimes E_j :=
 \{(t,x) \mapsto t^{-q_j^\pm}e(x) \mid e\in E_j\}.$$ 
Moreover we set 
$$ \wh{\scrE}_{q_0^\pm} =\begin{cases}t^{q_0^\pm}\otimes E_0&\quad: n>1\\
     1\otimes E_0 + \log t\otimes  E_0&\quad: n=1\end{cases}.
$$
For $q$ different from every $q_j^\pm$ we set $\wh{\scrE}_q=\{0\}$. 
Let us also introduce the interval  
$$I_0 = \Big(\frac{n+1}{2}-2, \frac{n+1}{2}\Big)=
    \Big(\frac{n-1}{2}-1, \frac{n-1}{2}+1\Big)$$
and  the translated intervals 
\begin{eqnarray}\label{eq:intervall}
 I_\gamma = I_0-\gamma,\qquad\gamma\in\rz. 
\end{eqnarray}

By Theorem \ref{thm:max_domain1}, the maximal domain of $\widehat \Delta$ is as follows. 

\begin{proposition}\label{prop:5.1}
The maximal extension of $\widehat\Delta$ as an unbounded operator on $\calK^{s,\gamma}_p(X^\wedge)$ has the domain
$$\scrD_{\max}(\widehat\Delta) = \calK^{s+2, \gamma+2}_p (X^\wedge) 
\oplus \mathop{\mbox{\LARGE$\oplus$}}_{q\in I_\gamma} \omega\wh{\scrE}_q.$$
In view of \eqref{eq:gamma}, the minimal domain is 
$\calK^{s+2, \gamma+2}_p (X^\wedge)$.  
\end{proposition}

We next consider an extension $\widehat {\underline \Delta}$ of $\widehat \Delta$ 
with domain
\begin{eqnarray}\label{eq:domain}
\scrD(\widehat {\underline \Delta}) = \calK^{2, \gamma+2}_p (X^\wedge) 
\oplus \mathop{\mbox{\LARGE$\oplus$}}_{q\in I_\gamma} \omega\ulwhE_q,
\end{eqnarray}
where $\ulwhE_q$ is a subspace of $\wh{\scrE}_q$. 
For $n=1$ and $q=0$ we confine ourselves to the choices $\ulwhE_0=\{0\}$,
$\ulwhE_0=1\otimes E_0$ or $\ulwhE_0=\wh\scrE_0$. 
We define the spaces $\ulwhE_q^\perp$ as follows:
For $j\not=0$ or $n>1$ we write  
$\ulwhE_{q^\pm_j}=t^{-q_j^\pm}\otimes \underline E_j$ 
with a subspace $\underline E_j$ of $E_j$ and let 
\begin{eqnarray*}
\ulwhE_{q_j^\pm}^\perp= t^{-q_j^\mp}\otimes \underline E_j^\perp
\end{eqnarray*} 
with the orthogonal complement $\underline E_j^\perp$ of $\underline E_j$ in $E_j$ 
with respect to the scalar product on $L^2(X)$. 
For $j=0$ and $n=1$ (i.e. $q_0^\pm=0$), we let 
$\ulwhE_0^\perp = \wh{\scrE}_0$, if $\ulwhE_0=\{0\}$,   
$\ulwhE_0^\perp = \{0\}$, if $\ulwhE_0=\wh{\scrE}_0$ and 
$\ulwhE_0^\perp = \ulwhE_0$, if $\ulwhE_0= 1\otimes E_0$.
For every other $q$ set $\ulwhE_q^\perp=\{0\}$. 

Just as in  \cite[Theorem 5.3]{SchrSe} we find that

\begin{proposition}\label{prop:5.2}
In case $p=2$, the domain of the adjoint $\underline{\widehat\Delta}^* $ of the extension 
$\underline{\widehat\Delta} $ with domain \eqref{eq:domain} is
$$\scrD( \underline{\widehat\Delta}^*) = \calK^{-s+2, -\gamma+2}_2 (X^\wedge) 
\oplus \mathop{\mbox{\LARGE$\oplus$}}_{q\in I_\gamma} \omega\underline{\widehat\scrE}_q^\perp.
$$
\end{proposition}

Theorems 5.6 and 5.7 in \cite{SchrSe} then imply the following result: 

\begin{theorem}\label{thm:5.3}
Let $|\gamma| <(n+1)/2$ satisfy \eqref{eq:gamma} and let 
$\underline{\widehat\Delta}$ be an extension with domain as in \eqref{eq:domain},
where the spaces $\underline{\widehat\scrE}_{q}$ are chosen such that: 
\begin{enumerate}
\item $\underline {\widehat{\scrE}}_q^\perp =\underline{\widehat  \scrE}_{n-1-q}$ for 
  $\displaystyle q\in I_\gamma\cap I_{-\gamma}$, 
\item $\underline{\widehat{\scrE}}_q = {\widehat{\scrE}}_q$ for $\gamma\ge 0$ and $q\in I_\gamma\setminus I_{-\gamma}$,
\item $\underline{\widehat{\scrE}}_q = \{0\}$ for $\gamma\le 0$ and $q\in I_\gamma\setminus I_{-\gamma}$. 
\end{enumerate}
Then $\underline{\widehat\Delta}$ satisfies \rm{(E3)} for every sector $\Lambda\subset \cz \setminus \rz_+$: There exists a $C\ge0$ such that 
$$\|\lambda ( \lambda -\widehat{\underline\Delta})^{-1}\|_{\scrL(\calK^{0,\gamma}_2(X^\wedge))}\le C,
 \qquad 0\not=\lambda \in \Lambda.$$
\end{theorem}

In principle, the associated domains of $\Delta$ can be determined as described in Section \ref{subsec:example}. 
But since the eigenvalues $\lambda_j$ are not known explicitly, it seems not feasible to 
write down a formula for the domains. 
See, however, the following Section \ref{sec:porous} for a more specific situation.

\subsection{The porous medium equation on manifolds with warped cones}\label{sec:porous}
The porous medium equation is the quasilinear diffusion equation
\begin{eqnarray}\label{eq:PME}
u'(t)-\Delta(u^{m}(t))=f(u,t),\  t\in(0,T_0],
\end{eqnarray}
with initial condition  $u(0)=u_{0}$. 
It describes the flow of a gas in a porous medium; here $u$ is the density of the gas, $m>0$ and $f$ is a forcing term which we assume for simplicity holomorphic in $u$ and Lipschitz in $t$. 

In \cite{RoidosSchrPME} the porous medium equation has been studied on a manifold with straight conical singularities. 
We will next show how this analysis can be extended to the case of warped cones.  

The setting is the same as in \cite{RoidosSchrPME}: We let $\overline \varepsilon = -q_1^->0$ and fix $\gamma$ with 
\begin{eqnarray}\label{eq:gamma2}
\frac{n-3}2<\gamma<\frac{n-3}2+\min \{\overline\varepsilon, 2\Big\}.
\end{eqnarray}
Then none of the poles $q_j^\pm$ lies on the line $\re z = (n+1)/2-\gamma-2$. Condition (E2) requires that also no pole lies on $\re z = (n+1)/2-\gamma$. In case $n\ge 3$ this is automatically true, since then $q_1^+>2$. In case $n=1$ or $n=2$ we additionally require it. 

In view of \eqref{eq:gamma2} we  write $\gamma = (n-3)/2+\delta$ for some 
$0<\delta <\min\{\overline\varepsilon,2\}$. Then 
  $$I_\gamma =(-\delta, 2-\delta), \qquad I_{-\gamma} = (n-3+\delta, n-1+\delta).$$ 
So $I_\gamma$ always contains $q_0^-=0$, but none of the $q_j^-$ for $j>0$, 
whereas $I_{-\gamma}$ contains $q_0^+$, 
but none of the $q^+_j$ for $j>0$. Moreover, $I_{\gamma}\cap I_{-\gamma}$ contains at most the elements 
$q^-_0=0$ and $q^+_0=n-1$.
In fact, for  $n=1$ we have $I_\gamma\cap I_{-\gamma}=\{0\}$. 
For $n=2$, the intersection contains both $q_0^-=0$ and $q_0^+=1$, provided $\delta<1$, else it is empty.
For $n\ge3$, the intersection is always empty. 

\subsubsection{The space $\scrE_0$}

We shall use the notation from Section \ref{sec:02.2} and from the beginning of Section \ref{sec:warped}, 
in particular \eqref{eq:f0} and \eqref{eq:f1}. 
Let us analyze the space $\scrE_0=\theta_0^{-1}\wh{\scrE}_0$ associated with $\Delta$. 

The principal part of $f_0(z)^{-1}$ in $z=q_j^\pm$ is 
 $$(\Pi_{q_j^\pm}f_0^{-1})(z)=\pm\frac{\pi_j}{q_j^+-q_j^-}(z-q_j^\pm)^{-1},\qquad j\ge 1;$$
for $j=0$ the same formula holds in case $n\ge2$, while for $n=1$ the principal part in $z=0$ is $\pi_0 z^{-2}$. 

\textbf{The case $\mathbf{n\ge 2}$:} If $\delta<1$ then $G_0=G_0^{(0)}$ by definition. 
For $\delta\ge1$, the fact that $\dot\Delta_0$ is a differential operator without constant term implies that $\dot\Delta_0\pi_0=0$ $($recall that $\pi_0$ is 
the projection onto the space $E_0$ of locally constant functions on $\partial\bz)$, and so
\begin{align*}
 g_1(z)\Pi_{0}((f_0^{-1})(z)\wh{u}(z))&=-f_0(z-1)^{-1}(\dot\Delta_0-\dot H(0)z)\frac{\pi_0\wh{u}(0)}{n-1}\frac{1}{z}\\
 &=-f_0(z-1)^{-1}\dot H(0)\frac{\pi_0\wh{u}(0)}{n-1}.
\end{align*}
As $f_0(z-1)^{-1}$ is holomorphic in $z=0$, 
$G_0^{(1)}=0$ and $G_0=G_0^{(0)}$, again. Thus 
 $$\wh{\scrE}_0=\scrE_0=1\otimes E_0.$$
\textbf{The case $\mathbf{n=1}$:} We calculate 
\begin{align*}
G_0^{(0)} u &= \int_{|z|<\varepsilon} t^{-z}\Pi_0(f_0(z)^{-1} \widehat u(z))\, dz\\
&= \int_{|z|<\varepsilon} t^{-z}
\pi_0\left(\frac{\widehat u(0)}{z^2}+\frac{\widehat u'(0)}{z}\right)\,dz=
\log t \ \pi_0\widehat u(0) +\pi_0\wh{u}^\prime(0),  
\end{align*}
showing that 
 $$\widehat{\scrE}_0 = \big\{e_0 + e_1\log t\mid e_0,e_1\in E_0\big\}=1\otimes E_0+ \log t\otimes E_0.$$
By definition, $G_0=G_0^{(0)}$ for $\delta<1$, while for $\delta\ge1$,  similarly as before, 
\begin{align*}
g_1(z)(\Pi_{0}f_0^{-1}\widehat u)(z)
=&-f_0(z-1)^{-1}(\dot\Delta_0-\dot H(0)z)\pi_0\left(\frac{\widehat u(0)}{z^2}+\frac{\widehat u'(0)}{z}\right)\\
=&f_0(z-1)^{-1}\dot H(0)\pi_0\left(\frac{\widehat u(0)}{z}+\widehat u'(0)\right),
\end{align*}
By definition of $\overline \varepsilon$, $f_0^{-1}$ is holomorphic in $z=-1$. Hence
$$G_0^{(1)}u = t a(x) \pi_0\widehat u(0),\qquad a(x):=f_0^{-1}(-1) \dot H(0)\in\scrC^\infty(\partial\bz).$$
We conclude that  
\begin{align*}
\scrE_0 =
\begin{cases}
 \big\{e_0 + e_1(\log t +t a(x))\mid e_0,e_1\in E_0\big\}, &\delta\ge1 \\[1mm]
\widehat\scrE_0,& \delta<1.
 \end{cases}
\end{align*} 
The isomorphism $\theta_0:\scrE_0\to\wh{\scrE}_0$ is the identity map in case $\delta<1$, otherwise 
 $$\theta_0(e_0 + e_1(\log t +t a(x)))=e_0 + e_1\log t,\qquad e_0,e_1\in E_0.$$

\subsubsection{A closed extension of $\Delta$}

In the following, we consider the closed extension $\underline \Delta$ of the Laplacian in 
$\calH^{s,\gamma}_p(\bz)$ defined by 
\begin{equation}\label{eq:extD}
 \scrD(\underline\Delta) = \calH_p^{s+2,\gamma+2}(\bz)\oplus \omega\underline{\scrE}_0,
 \qquad \underline{\scrE}_0:=1\otimes E_0. 
\end{equation}
By the relations obtained in the previous subsection we find that the associated extension $\wh{\underline{\Delta}}$ 
of the model cone operator $\wh{\Delta}$ is given by 
\begin{equation}\label{eq:extDhut}
 \scrD(\wh{\underline\Delta}) = \calK_p^{s+2,\gamma+2}(X^\wedge)\oplus \omega\underline{\wh{\scrE}}_0,
 \qquad \underline{\wh{\scrE}}_0=\underline{\scrE}_0=1\otimes E_0. 
\end{equation}

\begin{proposition}\label{prop:hatDelta}
$-\underline \Delta$ satisfies the 
ellipticity conditions $\mathrm{(E1)-(E3)}$ of  Section $\mathrm{\ref{sec:03}}$. 
Moreover, its spectrum is a subset of $\overline \rz_+$. 
\end{proposition}
\begin{proof}
By the choice of $\gamma$, $-\underline \Delta$ satisfies (E1) and (E2). (E3) holds, since 
$\wh{\underline\Delta}$ with the domain \eqref{eq:extDhut} satisfies the assumptions of Theorem \rm{\ref{thm:5.3}}.
The same argument as for the proof of \cite[Theorem 4.1]{RoidosSchr} shows that the spectrum is a subset of $\overline\rz_+$.
\end{proof} 

\forget{
\begin{lemma}\label{lem:hatDelta} The choice 
$$\calD(\underline{\widehat\Delta} ) = \calK^{2,\gamma+2}_2(X^\wedge)\oplus \omega\underline{\widehat\scrE}_0,\qquad 
    \underline{\widehat\scrE}_0=\underline{\scrE}_0,$$ 
satisfies the assumptions in Theorem \rm{\ref{thm:5.3}}, and thus \rm{(E3)} holds.
\end{lemma}
\begin{proof}
Indeed, for $\gamma\ge0$ the set $I_\gamma\setminus I_{-\gamma}$ does not 
contain any of the $q_j^\pm$, whereas for $\gamma<0$ we have 
$\underline{\widehat\scrE}_{\sigma}=\{0\}$ for all $\sigma$ in 
$I_\gamma\setminus I_{-\gamma}$.

It remains to observe that the preimage of $\underline{\widehat\scrE}_{0}$ 
under the isomorphism $\theta_0$ coincides with $\underline{\scrE}_0$, i.e., 
consists of the constant functions only. 
To this end we note that according to \eqref{eq:gsigma}, 
$G_0 = G_0^{(0)}$ for $\delta<1$. 
For $\delta\ge 1 $ on the other hand, the operator $G_0^{(1)}$ is zero 
by the considerations at the end of Section \ref{sect:domain}, 
since then $f_0^{-1}$ is holomorphic near $z=-1$. 
\end{proof}
}

\begin{theorem}\label{thm:delta}
Let $c>0$, $1<p<+\infty$ and $s\in \rz$. 
Then $c-\underline \Delta$ has a bounded $H_\infty$-calculus on $\calH^{s,\gamma}_p(\bz)$.
\end{theorem}

Note that here, in contrast to Theorem \ref{thm:hinfty}, also negative $s$ are allowed. 

\begin{proof}[Proof of Theorem $\ref{thm:delta}$]
Proposition \ref{prop:hatDelta} and Theorem \ref{thm:hinfty} show the existence of a bounded $H_\infty$-calculus  for $s\ge 0$. 

To cover negative values of $s$, we shall show that the adjoint operator $(-\underline\Delta)^*$ also satisfies the assumptions of 
Theorem \ref{thm:hinfty}, hence admits a bounded $H_\infty$-calculus on $\calH^{-s,-\gamma}_{p^\prime}(\bz)$. 
Thus its adjoint  $(-\underline\Delta)^{**}=-\underline\Delta$ admits a bounded $H_\infty$-calculus on $\calH^{s,\gamma}_p(\bz)$. Clearly, $(-\underline\Delta)^*$ satisfies (E1) and  (E2), so it suffices to check (E3).
We consider separately the cases $n\ge3$, $n=2$, and $n=1$. 

\textbf{The case $\mathbf{n\ge 3}$:} The operator $\underline\Delta$  coincides with the maximal extension of $\Delta$. 
Hence its adjoint is the minimal extension with domain $\calH^{2,-\gamma+2}_{2}(\bz)$. Accordingly, the model cone operator associated with 
$(\underline\Delta)^*$ is the minimal extension of $\wh{\Delta}$ and coincides with the adjoint of the maximal extension of $\wh{\Delta}$, i.e., 
with the adjoint of $\wh{\underline\Delta}$. Hence it is clear that $(-\underline\Delta)^*$ satisfies (E3). 

\textbf{The case $\mathbf{n= 2}$:} Recall that then $\underline{\scrE}_{0}=\scrE_0$. 
In case $\delta>1$, the only pole of the inverted conormal symbol in $I_\gamma=(-\delta,2-\delta)$ is $q_0^-=0$. 
In this case, $\underline{\Delta}$ coincides with the maximal extension of $\Delta$ and we can argue as before. 
So let us assume $\delta<1$ $($the case $\delta=1$ is excluded by the assumptions$)$.
Then the poles in $I_\gamma$ are 
$q_0^-=0$, $q_0^+=1$ and possibly a finite number of  $q_j^+$, $j=1,\ldots,N$, which are larger than 1 and 
smaller than $2-\delta$. Then $I_{-\gamma}=(-1+\delta,1+\delta)$ contains the poles $q_0^-=0$, 
$q_0^+=1$, and $q_j^-$, $j=1,\ldots,N$. Now write 
 $$\scrD_{\gamma,\max}(\Delta):=\calH^{2,\gamma+2}_2(\bz)\oplus \omega\scrE_{\gamma,\max},\qquad 
     \scrE_{\gamma,\max}:=\mathop{\mbox{\LARGE$\oplus$}}_{\sigma\in I_\gamma} \scrE_\sigma,$$
where $\scrE_\sigma=\theta^{-1}_\sigma\wh{\scrE}_\sigma$. 
Analogously  define  $\scrD_{-\gamma,\max}(\Delta)$. 
Then it is known that 
 $$[u,v]_\Delta:=(\Delta u,v)_{0,0}-(u,\Delta v)_{0,0}$$ 
yields a non-degenerate pairing on $\scrD_{\gamma,\max}(\Delta)\times \scrD_{-\gamma,\max}(\Delta)$ 
which vanishes whenever one of the entries belongs to the corresponding minimal domain, see e.g. \cite[Section 3]{GiMe}. 
Hence we obtain a non-degenerate pairing 
 $$[u,v]^\prime_\Delta:=[\omega_0 u,\omega_1 v],
     \qquad (u,v)\in \scrE_{\gamma,\max}\times \scrE_{-\gamma,\max},$$ 

which does not depend on the choice of the cut-off functions $\omega_0$ and $\omega_1$.    
Moreover, denoting by  ${\scrE}^\perp_0$ the space orthogonal to ${\scrE}_0$ with respect 
to $[\cdot,\cdot]^\prime$, the domain of $(\underline\Delta)^*$ is given by  
$\calH^{2,-\gamma+2}_2(\bz)\oplus\omega{\scrE}^\perp_0$. 

We shall now verify that $\scrE_0^\perp=\scrE_{q_0^-}\oplus\scrE_{q_1^-}\oplus\ldots\oplus \scrE_{q_N^-}$. 
Since $\scrE_0$ has dimension $d:=\mathrm{dim}\,E_0$, $\scrE^\perp_0$ has co-dimension $d$ in  $\scrE_{-\gamma,\max}$. 
Hence it suffices to show that $\scrE_{q_j^-}\subset\scrE_0^\perp$ for every $j=0,\ldots,N$ $($recall that 
$\mathrm{dim}\,\scrE_{q_0^+}=\mathrm{dim}\,\scrE_1=\mathrm{dim}\,\wh{\scrE}_1=d$$)$. 
According to Theorem \ref{thm:theta}, the elements of $\scrE_{q_j^-}$ are of the 
form $t^{-q_j^-}e_j $, $e_j \in E_j$. 

Let $u=c\in\scrE_0$ and $v=e_jt^{-q_j^-}\in\scrE_{q_j^-}$. Fixing $\omega_1$ and choosing $\omega_0$ so that 
$\omega_1\equiv 1$ on the support of $\omega_0$, we find that 
\begin{equation}\label{eq:pairing}
 [u,v]^\prime_\Delta=(c\Delta\omega_0,e_jt^{-q_j^-})_{0,0}-(c\omega_0,\Delta(e_jt^{-q_j^-}))_{0,0}.
\end{equation}
The inner product is that of $L^2((0,1)\times X,t^nd\mu_t\,dt)$, $n=2$, where the measure $\mu_t$ refers to the Riemannian metric $h(t)$ 
on $X$ $($note that both $\omega_0$ and $\Delta\omega_0$ are supported in $[0,1))$. 
We can take $\omega_0(t)=\omega(t/\eps)$ for a fixed cut-off function $\omega$ and arbitrary 
$\eps>0$ sufficiently small. Since then 
 $$|\omega(t/\eps)\cdot\Delta(e_jt^{-q_j^-})|\le \omega_1(t)|\Delta(e_jt^{-q_j^-})|\in L^1,$$
the second term on the right-hand side of \eqref{eq:pairing} converges to $0$ as $\eps\to0$ by Lebesgue's theorem on dominated convergence. 
Moreover,
\begin{eqnarray*}
\Delta (\omega(t/\eps))&=&
t^{-2}[((t\partial_t)^2\omega)(t/\eps)+(1+H(t))(t\partial_t\omega)(t/\eps)]\\
&=&
\eps^{-2}[\varphi_1(t/\eps)+H(t)\varphi_2(t/\eps)]
\end{eqnarray*}
for suitable $\varphi_1,\varphi_2\in\scrC^\infty_\comp((0,1)).$
Thus, after the change of variables $s=t/\eps$,  the first term on the right-hand side of \eqref{eq:pairing} is  
 $$(c\Delta(\omega(t/\eps)),e_jt^{-q_j^-})_{0,0}
     =\eps^{1-q_j^-}\int_0^1\int_Xc(x)[\varphi_1(s)+H(\eps s)\varphi_2(s)]\overline{e_j(x)}s^{-q_j^-}s^2d\mu_{\eps s}\,ds.$$
Since all $q_j^-$ are non-positive, this expression vanishes as $\eps$ tends to $0$. 

Hence $\scrE_0^\perp$ is as claimed and $(-\underline\Delta)^*$ satisfies (E3) by Theorem \ref{thm:5.3}. 

\textbf{The case $\mathbf{n= 1}$:} The argument is very similar to that for $n=2$. 
The interval $I_\gamma$ contains the $($double$)$ pole 
$q_0^+=q_0^-=0$ and possibly a finite number of poles $q_j^+$, $j=1,\ldots,N$. Then $I_{-\gamma}$ contains the poles $0$ 
and $q_j^-$, $j=1,\ldots,N$. One shows now that 
$\underline{\scrE}_0^\perp=\underline{\scrE}_0\oplus \scrE_{q_1^-}\oplus\ldots\oplus \scrE_{q_N^-}$. 
To this end, let first $v\in \scrE_{q_j^-}$ for $j\ge1$. Since the poles in all the $q^-_j$, $j\ge1$, are simple, Theorem  \ref{thm:theta} implies that  $v$ has the form 
 $$v(t,x)=a(x)t^{-q_j^{-}}+t^{1-q_j^{-}}\sum_{k=0}^1 a_k(x)\log^kt$$
for certain functions $a,a_0,a_1$. 
Following the above calculations, 
using the measure $t\,d\mu_t\,dt$, it is easy to see that $v$ is perpendicular to $\underline{\scrE}_0$.  
If $u=c_0,v=c_1\in\underline{\scrE}_0$ are two locally constant functions, the above argument shows that 
\begin{align*}
 [u,v]^\prime_\Delta
&=\lim_{\eps\to0}\int_0^1\int_Xc_0(x)\overline{c_1(x)}
   s^{-2}[(s\partial_s)^2\omega(s)+H(\eps s)(s\partial_s)\omega(s)]s\,d\mu_{\eps s}\,ds\\
&=\int_Xc_0(x)\overline{c_1(x)}\,d\mu_0\,\int_0^1 (s\partial_s)^2\omega(s)\,\frac{ds}s
    =0.
\end{align*}
Hence $\underline{\scrE}_0\subset\underline{\scrE}_0^\perp$. Since both $\underline{\scrE}_0^\perp$ and 
$\underline{\scrE}_0\oplus \scrE_{q_1^-}\oplus\ldots\oplus \scrE_{q_N^-}$ have co-dimension $d=\mathrm{dim}\,E_0$ in $\scrE_{-\gamma,\max}$, 
the desired equality of both spaces follows. An application of Theorem \ref{thm:5.3} shows (E3). 
\end{proof}

In the previous proof we have verified that the model cone operator associated with $\underline{\Delta}^*$ coincides 
with the adjoint of the model cone operator associated with $\underline{\Delta}$, i.e., 
$\widehat{(\underline{\Delta}^*)}=(\widehat{\underline{\Delta}})^*$. Though this identity appears to be quite natural it does not hold 
true, in general, for closed extensions of arbitrary cone differential operators. In fact, here is a counter-example$:$ 

\begin{example}\rm
Let $0\not=\alpha\in\rz$. On the half-axis $\rz_+=(0,+\infty)$ let us consider  
 $$A_\alpha=\partial_t^2-\alpha\partial_t.$$
We will analyze the closed extensions in $L^2(\rz_+)=\calK^{0,0}(\rz_+)$. 
We can represent this operator in the form 
 $$A_\alpha=t^{-2}\big(f_0(-t\partial_t)+tf_1(-t\partial_t)\big),\quad f_0(z)=z(z+1), \quad f_1(z)=\alpha z.$$
In particular, the associated model cone operator is 
 $$\wh{A}_\alpha=\wh{A}=(-t\partial_t)(1-t\partial_t)=\partial_t^2;$$
obviously it does not depend on $\alpha$. 
Its maximal domain is $\scrD_{\max}(\wh{A})=\calK^{2,2}(\rz_+)\oplus \omega\wh{\scrE}$, where 
 $$\wh{\scrE}=\wh{\scrE}_0\oplus \wh{\scrE}_{-1},\qquad 
     \wh{\scrE}_0=\spk{1},\quad \wh{\scrE}_{-1}=\spk{t}; $$
throughout this example we shall write $\spk{u}=\mathrm{span}\{u\}$ and $\spk{u,v}=\mathrm{span}\{u,v\}$. 
Let us now determine the maximal domain of $A_\alpha$. To this end we apply Theorem $\ref{thm:theta}$. 

Clearly $G^{(1)}_{-1}=0$; hence $\scrE_{-1}=\wh{\scrE}_{-1}$ and $\theta_{-1}$ is the identity operator. 
A straightforward calculation shows that 
 $$(G^{(1)}_0u)(t)=\alpha t\wh{u}(0),\qquad u\in\scrC^\infty_{\mathrm{comp}}(\rz_+).$$
It follows that $\scrE_0=\spk{1+\alpha t}$ and that $\theta_0:\scrE_0\to\wh{\scrE}_0$ is given by $\theta(1+\alpha t)=1$. 
This shows that 
 $$\scrD_{\max}(A_\alpha)=\calK^{2,2}(\rz_+)\oplus \omega\scrE,\qquad \scrE=\spk{1,1+\alpha t}=\spk{1,t}.$$ 
Though $\scrE$ does not depend on $\alpha$, the isomorphism $\Theta=\Theta_\alpha$ of \eqref{eq:Grassmannian} does; in fact,  
 $$\Theta(\spk{a+bt})=\spk{a+(b-a\alpha)t},\qquad a,b\in\cz.$$
Since the formal adjoint of $A_\alpha$ is $A_{-\alpha}$, the  pairing $\scrE\times\scrE\to\cz$ that determines the domain 
of the adjoint of a closed extension of $A_\alpha$ is given by 
\begin{align*}
 [u,v]_\alpha^\prime&=(A_\alpha (\omega u),\omega v)-(\omega u,A_{-\alpha}(\omega v))\\
 &=((\omega u)^{\prime\prime},\omega v)-(\omega u,(\omega v)^{\prime\prime})
  -\alpha\big(((\omega u)^\prime,\omega v)+(\omega u,(\omega v)^\prime)\big)
\end{align*}
with the inner product $(\cdot,\cdot)$. in $L^2(\rz_+)$. We thus find the formula 
 $$[a+bt,c+dt]_\alpha^\prime=a\overline{d}-b\overline{c}+\alpha a\overline{c},\qquad a,b,c,d\in\cz.$$
Analogously, the pairing for the model cone operator is given by this formula with $\alpha=0$. 

Now let $\ulA$ be the closed extension of $A_\alpha$ with domain determined by $\underline{\scrE}=\spk{1+\alpha t}$. 
Then for the orthogonal space we find $\underline{\scrE}^\perp=\spk{1}$. The extension of $A_{-\alpha}$ with domain 
determined by $\underline{\scrE}^\perp$ is the adjoint of $\ulA$. The model cone operator $\ulwhA$ is determined by 
$\Theta_\alpha(\underline{\scrE})=\Theta_\alpha(\spk{1+\alpha t})=\spk{1}=\underline{\wh\scrE}$. Note that 
$\underline{\wh\scrE}$ is perpendicular to itself, i.e., $\ulwhA$ is self-adjoint. However, the model cone operator 
of $\ulA^*$ is the extension of $\wh{A}$ determined by 
$\Theta_{-\alpha}(\underline{\scrE}^\perp)=\Theta_{-\alpha}(\spk{1})=\spk{1-\alpha t}\not=\spk{1}$. Thus 
$\wh{(\ulA^*)}\not=(\ulwhA)^*$.  
\end{example}

\subsubsection{Short time existence for the porous medium equation}

We shall apply the Theorem of Cl\'ement and Li with $X_0= \calH^{s,\gamma}_p(\bz)$ and 
$X_1=\calH^{s+2,\gamma+2}_p(\bz)\oplus \omega \underline\scrE_0$.

Choose $1<p,q<\infty$ so large that 
\begin{eqnarray}\label{eq:pq}
\frac{n+1}p+\frac2q < 1, \ \text{ and } \frac{n-3}2+\frac2q<\gamma.
\end{eqnarray}
Moreover, fix 
\begin{eqnarray}\label{eq:s}
s>-1+\frac{n+1}{p}+\frac{2}{q}.
\end{eqnarray}

With the same proof as for \cite[Theorem 6.1]{RoidosSchrPME} we  obtain the proposition, below. In \cite{RoidosSchrPME},  $\Delta$ was supposed to be the Laplacian with respect to the straight cone metric. The proof, however, does not use this geometric assumption but only the fact that, for $c>0$,  $c-\underline \Delta$ is  $R$-sectorial of angle $\theta$ for every $\theta\in (0,\pi)$. In the present case,  $R$-sectoriality is implied by the existence of the bounded $H_\infty$-calculus. 

\begin{proposition}
Let $s,\gamma,p$ and $q$ be  as above and let $v\in (X_0,X_1)_{1-\frac1q,q}$
satisfy $v\ge \alpha>0$ for some constant $\alpha$. Then, for every $\theta\in (0,\pi)$, there exists a constant $c>0$ such that $c-u\underline \Delta:X_1\to X_0$ is $R$-sectorial of angle $\theta$. 
\end{proposition}

\begin{remark}\rm
There are alternative ways to obtain this result. The key point is  the 
$R$-sectoriality of $\underline \Delta$, which we infer here from the bounded $H_\infty$-calculus. Alternatively, one might proceed as in \cite[Theorem 4.1]{Roidos16} or use the technically very difficult argument in  \cite[Section 5]{RoidosSchr}.
\end{remark}

\begin{theorem}\label{thm:solution}
Choose $s,\gamma,p$ and $q$ as in \eqref{eq:gamma2}, \eqref{eq:pq} and \eqref{eq:s}. Then for any strictly positive initial value
$u_0\in (X_0,X_1)_{1-1/q,q}$ there exists a $T>0$ such that  the porous medium equation  \eqref{eq:PME} has a unique solution 
$$u\in L^q\big(0,T; \calH^{s+2,\gamma+2}_p(\bz) \oplus \omega \underline\scrE_0\big)
\cap W^{1,q}\big(0,T; H^{s,\gamma}_p(\bz)\big).$$
\end{theorem}

\begin{proof}
We apply Cl\'ement and Li's theorem. 
The maximal regularity of the operator $u_0\underline\Delta$ follows from 
Proposition \ref{prop:MR}.
Properties (H1) and (H2) have been shown in \cite[Theorem 6.5]{RoidosSchr};  (H3) is trivially fulfilled,  since $g=0$.
\end{proof}


\bibliographystyle{amsalpha}

\end{document}